\documentclass[5p,times]{elsarticle}

\usepackage{graphicx}
\usepackage{amssymb,amsmath,amsthm}
\usepackage{url} 
\usepackage{float}
\usepackage{subcaption}
\usepackage{algorithm}
\usepackage{algpseudocode}
\usepackage{color}
\usepackage{xcolor} 

\newtheorem{theorem}{Theorem}[section]
\newtheorem{corollary}{Corollary}[theorem]

\newtheorem{definition}{Definition}[section]
\newtheorem{proposition}{Proposition}[section]

\DeclareMathOperator*{\argmin}{arg\,min}

\newcommand{\ignore}[1]{}

\begin{document}

\begin{frontmatter}
\title{
Random Sketching to Enhance the Numerical Stability of\\
Block Orthogonalization Algorithms for $s$-step GMRES}

\author[snl]{Ichitaro Yamazaki}
\author[snl]{Andrew J. Higgins}
\author[snl]{Erik G. Boman}
\author[temple]{Daniel B. Szyld}

\fntext[snl]{Sandia National Laboratories, Albuquerque, New Mexico, U.S.A}
\fntext[temple]{Temple University, Philadelphia, Pennsylvania, U.S.A}

\begin{abstract}
We integrate random sketching techniques into block orthogonalization schemes needed for $s$-step GMRES.
The resulting, one-stage and two-stage, block orthogonalization schemes generate the basis vectors whose overall orthogonality error is bounded by machine precision as long as each of the corresponding block vectors are numerically full rank.
We implement these randomized block orthogonalization schemes using standard distributed-memory linear algebra kernels for $s$-step GMRES available in the Trilinos software packages.
Our performance results on the Perlmutter supercomputer (with four NVIDIA A100 GPUs per node) demonstrate that these randomized techniques can enhance the numerical stability of the orthogonalization and overall solver, without a significant increase in the execution time. 
\end{abstract}
\end{frontmatter}

\section{Introduction}

Generalized Minimum Residual (GMRES)~\cite{Saad:1986} is a popular subspace projection method for iteratively solving a large linear system of equations as it computes the approximate solution that minimizes the residual norm in the generated Krylov subspace.
To compute the approxmate solution,
GMRES generates the orthonormal basis vectors of its projection subspace based on two main computational kernels: 1) Sparse-Matrix Vector multiply (SpMV), typically combined with a preconditioner, and 2) orthogonalization.

Though GMRES is a robust iterative method for solving general linear systems,
the performance of these two kernels can be limited by communication costs (e.g., the cost of moving data through the local memory hierarchy and between the MPI processes). For instance, on a distributed-memory computer, to orthogonalize a new basis vector at each iteration,
GMRES requires global reduces among all the MPI processes
and performs its local computation based on either BLAS-1 or BLAS-2 operations. Hence, though the breakdown of the iteration time depends on the target hardware architecture and problem properties (e.g., the sparsity structure of the matrix and the preconditioner being used), orthogonalization can become a significant part of the iteration time, especially when scalable implementations of SpMV and preconditioner are available.

To improve the performance of the orthogonalization and of GMRES,
communication-avoiding (CA) variants of GMRES~\cite{Carson:2015,Hoemmen:2010}, based on $s$-step methods~\cite{Sturler:1995,Joubert:1992}, have been proposed.
These variants generate a set of $s$ basis vectors at a time,
utilizing two computational kernels:
1) the Matrix Powers Kernel (MPK) to generate the set of $s$ Krylov vectors by applying SpMV and preconditioner $s$ times, followed by
2) the Block Orthogonalization Kernel that orthogonalizes the set of $s+1$ basis vectors at once.
This provides the potential to reduce the communication cost of generating the $s$ basis vectors by a factor of $s$ (requiring the global reduce only at every $s$ step and using BLAS-3 for performing most of the local computation).
This is a very attractive feature, especially on currently available GPU clusters, 
where communication can be significantly more expensive compared to computation.

Since the potential speedup from the block orthogonalization is limited by the small step size $s$ required to maintain the numerical stability of MPK (e.g., to keep the $s+1$ basis vectors numerically full rank), a two-stage variant of block orthogonalization was proposed~\cite{Yamazaki:2024}. In order to maintain the well-conditioning of the basis vectors, at every $s$ steps, the first stage of this orthogonalization scheme pre-processes the block of new $s+1$ basis, while the full orthogonalization is delayed until enough number of basis vectors, $\widehat{s}+1$, are generated to obtain the higher performance. This improves the performance of the block orthogonalization process while using the small step size $s$.


Block orthogonalization consists of 
1) \emph{inter} block orthogonalization
to orthogonalize a new block of vectors against the already-orthogonalized blocks of vectors
and 2) \emph{intra} block orthogonalization 
to orthogonalize among the vectors within the new block.
For the inter-block orthogonalization, Block Classical Gram-Schmidt with re-orthogonalization (BCGS twice, or BCGS2) obtains good performance on current hardware architectures because it is based entirely on BLAS-3.
For robustness and performance of the overall block orthogonalization, the critical component is the algorithm used for the first intra-block orthogonalization~\cite{Barlow:2024}. In this paper, we consider the use of CholQR~\cite{Stath:2002} twice (CholQR2) as our intra-block orthogonalization, which is based mainly on BLAS-3.
Unfortunately,
though the above combinations of the algorithms performs well on current hardware, $s$-step basis vectors can be ill-conditioned, and CholQR2 can fail
when the condition number of the block of $s+1$ vectors is greater than the reciprocal of the square-root of machine epsilon (see Section~\ref{sec:cholqr2}).

To enhance the numerical stability of the above block orthogonalization schemes and of the overall $s$-step GMRES solver, we integrate random sketching techniques. Theoretical studies of such randomized schemes for the intra-block orthogonalization have been established in two recent
papers~\cite{Balabanov:2022,Higgins:2025}. We extend these studies to develop randomized BCGS2 schemes that generate the blocks of basis vectors whose overall orthogonality errors are bounded by machine epsilon.
We have initially presented our preliminary results of the current paper at 
the SIAM Conference on Parallel Processing for Scientific Computing (SIAM PP), 2024~\cite{siampp:2024}.

Our main contributions are:
\begin{itemize}
\item We integrate random sketching techniques into BCGS2 such that overall orthogonalization error is on the order of machine precision in both one-stage and two-stage frameworks as long as each of the corresponding block of $s+1$ and $\widehat{s}+1$ vectors are numerically full-rank, respectively.

\item We present numerical results to demonstrate the improved numerical stability using random sketching techniques (to pre-process the basis vectors) compared to the state-of-the-art deterministic algorithms (BCGS2 with CholQR2).

\item We implement Gaussian and Count sketching, and its combination, Count-Gauss sketching~\cite{Woodruff:2014}, for the $s$-step GMRES in Trilinos~\cite{Trilinos:2005,trilinos-website}, which is a collection of
          open-source software packages for supporting
          large-scale scientific and engineering simulation codes. Trilinos software stack
          allows the solvers, like $s$-step GMRES, to be portable to different computer architectures,
          using a single code base. In particular, our implementation of the random-sketching is based solely on standard distributed-memory linear algebra kernels (GEMM and SpMM), which are readily available in vendor-optimized libraries.

\item We study the performance of the block orthogonalization and $s$-step GMRES on the Perlmutter supercomputer at National Energy Research Scientific Computing (NERSC) center. Our performance results on up to 64 NVIDIA A100 GPUs show that random sketching has virtually no overhead to enhance the numerical stability of the one-stage BCGS2. Although it has a higher overhead due to the larger sketch size required for the two-stage algorithm,
the overhead became less significant as we increased the number of MPI processes. For example, the overhead was about 1.49$\times$ on 1 node, while it was about 1.19$\times$ on 16 nodes.
\end{itemize}

Table~\ref{tab:notation} lists the notation used in this paper.
In addition, we use $Q_{\ell:t}$ to denote the blocks column vectors of $Q$ with the block column indexes $\ell$ to $t$, while $q_{k:s}$ is the set of vectors with the column indexes $k$ to $s$. 
We then use the bold small letter $\mathbf{v}_j$ to denote the sketched version of the block vector $V_j$, e.g., $\mathbf{v}_j = \Theta^T V_j$.
Finally,
$[Q,V]$ is the column concatenation of $Q$ and $V$. For our numerical analysis,
we use $c_k(\epsilon, n,s)$ to represent a scalar constant that is in the order of the machine epsilon $\epsilon$
but also depends on the matrix dimensions $n$ and $s$. 

\begin{table}
\centerline{\footnotesize
\begin{tabular}{l|l}
notation    & description\\
\hline
 $n$        & problem size\\
 $m$        & subspace dimension \\
 $s$        & step size (for the first stage)\\
 $\widehat{s}$   & second step size (for the second stage and $s \le \widehat{s} \le m$)\\
 $v_k^{(j)}$ & $k$th basis vector within the $j$-th $s$ basis vectors\\
$V_j$       & $j$th $s$-step basis vectors including the starting vector, i.e.,\\
            & a set of $s+1$ vectors generated by MPK \\
            & $V_j = [v_{s(j-1)+1},v_{s(j-1)+2},\dots,v_{sj+1}]$\\
            & and $V_0 = [v_0]$ to simplify the notation\\
$\underline{V}_j$
            & same as $V_j$ except excluding the last vector, \\
            & which is the first vector of $V_{j+1}$, i.e.,\\
            & a set of $s$ vectors $\underline{V}_j = [v_{s(j-1)+1},v_{s(j-1)+2},\dots,v_{sj}]$\\
$\widehat{V}_j$  & $V_j$ after the first inter-block orthogonalization\\
$\overline{Q}_j$ & $V_j$ after the pre-processing stage\\
$\widehat{Q}_j$  & $V_j$ after the first intra-block orthogonalization\\
$Q_j$           & orthogonal basis vectors of $V_j$\\
$\Theta$        & sketch matrix\\
$\mathbf{v}_j$ & sketched version of the block vector $V_j$ (i.e., $\mathbf{v}_j := \Theta^T V_j$)\\
$\epsilon$  & machine epsilon\\
 $c_k(\epsilon, n,s)$ & a scalar constant in the order of $\epsilon$
                        but also depends on $n$ and $s$\\

$\kappa(V_j)$ & $\ell_2$-norm condition number of $V_j$ \\
$\| \cdot \| $ & $\ell_2$-norm
\end{tabular}}
\caption{Notation used in the paper.} \label{tab:notation}
\end{table}

\section{Related Work and Our Motivations}
In recent years, random sketching has been used
to improve the performance of Krylov solvers~\cite{Balabanov:2021:block,Balabanov:2022,Balabanov:2022:GMRES,Balabanov:2023:krylov}.
In contrast to the previous works that have focused on ``pseudo-optimal'' GMRES,
generating the basis vectors that are orthonormal with respect to the sketched
inner-product,
we use the random sketching to generate the well-conditioned basis vectors, but then explicitly generate the $\ell_2$-orthonormal basis vectors.
\begin{itemize}
\item
One reason for this is that except for one special case (i.e., two-stage with $\widehat{s}=m$, where $m$ is the restart length), we sketch only a part of the Krylov subspace (e.g., each panel or big panel of $s+1$ or $\widehat{s}+1$ basis vectors, respectively), requiring the $\ell_2$-orthogonality to ensure the overall consistency of the basis vectors over the restart loop.
\item
In addition, though the sketched norm is expected to be close to the original norm,
they could deviate from the $\ell_2$-norm in practice~\cite{Balabanov:2023:krylov}.
Though generating the $\ell_2$-orthonormal basis vectors requires additional cost
(both in term of computation and communication),
we expect the convergence of our implementation of sketched $s$-step GMRES to be the same as the original $s$-step GMRES, which is useful in practice.
\end{itemize}

Communication-avoiding (CA) variants of the tall-skinny Householder QR algorithm have been proposed~\cite{Demmel:2012} and its superior performance over the standard algorithm has been demonstrated~\cite{Anderson:2011}. In this paper, we focus on the performance comparison of the randomized algorithm against CholQR-like algorithms. Although with a careful implementation, CA Householder may obtain the performance close to CholQR, we believe CholQR, which is mostly based on standard BLAS-3 operation, as the baseline performance is beneficial. In addition, for $s$-step GMRES, it is convenient to explicitly generate the orthogonal basis vectors, where the CA variants require non-negligible performance overhead. 

The GPU performance of random sketching has been previously studied~\cite{Mary:2015,Swirydowicz:2023}.
This paper differs from these previous works since we focus on
the algorithmic development of numerically stable block orthogonalization schemes. We then integrate the resulting randomized block orthogonalization schemes into $s$-step GMRES, and study their performance impact on a GPU cluster.
We focus on the practical implementations of random sketching, using readily-available standard linear algebra kernels, though specialized kernels may improve the performance. 

Some of the sparse sketching techniques have the potential to reduce the computational complexity of the orthogonalization process. Unfortunately, we have not seen this performance benefit in our performance experiments, using the standard linear algebra kernels, while it may be possible by developing specialized implementations of sketching.
Nervelessness, our main focus is to enhance the numerical stability of the solver by integrating the random sketching techniques. Specifically, we develop randomized block orthogonalization algorithms that obtain $O(\epsilon)$ orthogonality errors, given the corresponding block vectors are numerically full-rank.

Since random sketching enhances the stability, it may be possible to use a larger step size for some matrices. However, it is often not feasible to tune the step size for each problem on a specific hardware (the largest step size to maintain the stability of MPK). Hence, we focus on improving the stability of the solver with the current default setup of Trilinos (i.e., $s=5$) allowing us to solve the problem, where the original algorithm failed, and study the required performance overhead. However, we will also discuss the computational and communication complexities of each algorithm in Section~\ref{sec:complexity}.

There have been significant advances in the theoretical understanding of $s$-step Krylov methods~\cite{Carson:2015} including the recent arXiv papers~\cite{Burke:2025,Carson2:025}. We will leave the potential integration and application of the specific random sketching techniques studied in this paper to those recent works as potential future studies.

\section{Background}
\label{sec:theory}
\setcounter{theorem}{1}

By $\Theta \in \mathbb{R}^{n\times\widehat{m}}$, we denote a  \textit{random sketch matrix}, with
$\widehat{m} \ll n$. The general concept of ``random sketching" is to apply this random sketch matrix $\Theta$ (generated using specific probability distributions) to a large matrix $V$ to obtain a ``sketched" matrix $\mathbf{v} = \Theta^T V$ that greatly reduces the row dimension of $V$ while preserving its fundamental properties, such as its norm and singular values, as much as possible.
In particular, the sketch matrix $\Theta$ is typically chosen to be a \textit{subspace embedding}, or a linear map to a lower dimensional space that preserves the $\ell_2$-inner product of all vectors within the subspace up to a factor of $\sqrt{1 \pm \mu}$ for some $\mu \in (0,1)$ \cite{Nakatsukasa:2021,Sarlos:2006}. Such embeddings also preserve $\ell_2$-norms in a similar way \cite{Balabanov:2022}. 

\begin{definition}[$\mu$-subspace embedding] \label{def:subspaceEmbedding}
Given $\mu \in (0,1)$, the sketch matrix $\Theta \in \mathbb{R}^{n\times\widehat{m}}$ is a \emph{$\mu$-subspace embedding} for the subspace $\mathcal{V} \subset \mathbb{R}^n$ if $\forall x, y \in \mathcal{V}$, 
\begin{equation}
    | \langle x,y \rangle - \langle \Theta^Tx,\Theta^Ty \rangle | \leq \mu \|x\|_2\|y\|_2. \label{eq:innerProdEmbedding}
\end{equation}
\end{definition}

Equation \eqref{eq:innerProdEmbedding} provides a straightforward relation between the sketching matrix and the preservation of the $\ell_2$-norm.

\begin{corollary} \label{cor:embeddingNorm}
If the sketch matrix $\Theta \in \mathbb{R}^{n\times\widehat{m}}$ is a \emph{$\mu$-subspace embedding} for the subspace $\mathcal{V} \subset \mathbb{R}^n$, then $\forall x \in \mathcal{V}$, 
\begin{equation}
    \sqrt{1-\mu}~\|x\|_2 \leq \|\Theta^Tx\|_2 \leq \sqrt{1+\mu}~\|x\|_2. \label{eq:normEmbedding} 
\end{equation}
\end{corollary}

Corollary \ref{cor:embeddingNorm} implies that we can also bound the singular values of a matrix $V$ by those of the sketched matrix $\Theta^TV$. Hence if $\Theta^TV$ is well conditioned, then so is $V$.

\begin{corollary} \label{cor:embeddingSingularVals}
If the sketch matrix $\Theta \in \mathbb{R}^{n\times\widehat{m}}$ is a \emph{$\mu$-subspace embedding} for the subspace $\mathcal{V} \subset \mathbb{R}^n$, and $V$ is a matrix whose columns form a basis of $\mathcal{V}$, then 
    \begin{align}
        (1+\mu)^{-1/2}~\sigma_{min}(\Theta^TV) &\leq \sigma_{min}(V) \leq \sigma_{max}(V) \label{eq:singValEmbedding} \\
        &\leq (1-\mu)^{-1/2}~\sigma_{max}(\Theta^TV). \nonumber 
    \end{align}
    Thus,
    \begin{equation}
        \kappa(V) \leq \sqrt{\frac{1+\mu}{1-\mu}}~\kappa(\Theta^TV). \label{eq:condEmbedding}
    \end{equation}
\end{corollary}
Proofs for Corollary \ref{cor:embeddingNorm} and \ref{cor:embeddingSingularVals} can be found in \cite{Balabanov:2022}. 

The limitation of $\mu$-subspace embedding presented in Definition~\ref{def:subspaceEmbedding} is that to ensure that the sketch matrix approximately preserves norms and inner products, one needs to know the subspace $\mathcal{V} \subset \mathbb{R}^n$ a priori. In contrast, to use sketching techniques in Krylov subspace methods efficiently, we need a sketch matrix that does not require complete prior knowledge of the subspace, since Krylov subspaces are generated as the algorithm iterates. This can be accomplished by using \emph{$(\mu, \delta, \widehat{s})$ oblivious $\ell_2$-subspace embeddings}~\cite{Balabanov:2022}.

\begin{definition}[$(\mu, \delta, \widehat{s})$ oblivious $\ell_2$-subspace embedding] \label{def:obliviousSubspaceEmbedding}
The sketch matrix $\Theta \in \mathbb{R}^{n\times\widehat{m}}$ is a \emph{$(\mu, \delta, \widehat{s})$ oblivious $\ell_2$-subspace embedding} if it is a $\mu$-subspace embedding for any fixed $\widehat{s}$-dimensional subspace $\mathcal{V} \subset \mathbb{R}^n$ with probability at least $1-\delta$.
\end{definition}

One concrete example of a $(\mu, \delta, \widehat{s})$ oblivious $\ell_2$-subspace embedding is $\Theta = \frac{1}{\sqrt{\widehat{m}}}G$ where $G \in \mathbb{R}^{n\times\widehat{m}}$ is a Gaussian matrix and the sketch size is given by $\widehat{m} = \Omega(\mu^{-2}\widehat{s})$~\cite{Martinsson:2020}. In practice, the sketch size may be chosen as $\widehat{m} \approx \widehat{s}/\mu^2$ ~\cite{Nakatsukasa:2021}. This relation allows a simple correspondence between the sketch size (or equivalenty the embedding dimension) $\widehat{m}$ and the subspace dimension $\widehat{s}$ for a given $\mu$. For instance, to achieve $\mu = 1/\sqrt{2}$, the sketch size of $\widehat{m} \approx 2\widehat{s}$ is sufficient,
though in principle, one could choose a different value of $\mu$ to construct a different sketch size. Other $(\mu, \delta, \widehat{s})$ oblivious $\ell_2$-subspace embeddings exist that can be stored in a sparse format, including sub-sampled randomized Hadamard and Fourier transforms (SRHT and SRFT, respectively), and ``sparse dimension reduction maps'' \cite{Balabanov:2022, Nakatsukasa:2021}.

\section{Block Orthogonalization for $s$-step GMRES}
\label{sec:sstep}

\begin{figure}[t]
\begin{center}
  \centerline{\fbox{\begin{minipage}[h!]{.9\linewidth}
    \footnotesize
    \input{codes/sstep-gmres}
  \end{minipage}}}
\end{center}
  \caption{
           Pseudocode of $s$-step GMRES where $[Q_j,R_j] = \mbox{qr}(Q,V_j)$ extends the QR factorization such that $Q R = V$ with $Q^TQ = I$ and upper-triangular $R$ with non-negative diagonals}
           \label{algo:sstep}
\end{figure}

Figure~\ref{algo:sstep} shows the pseudocode of $s$-step GMRES
for solving a linear system $Ax=b$ with a preconditioner $M^{-1}$,
which has been also implemented in Trilinos software framework \cite{Trilinos:2005,trilinos-website}. 
Though we focus on monomial basis vectors in this paper,
Trilinos also has an option to generate
Newton basis~\cite{Bai:1994} to improve the numerical stability of the basis vectors $V_j$
generated by the ``matrix-powers kernel'' (MPK).

Compared to the standard GMRES, $s$-step GMRES has the potential to reduce the communication cost 
of generating the $s$ basis vectors by a factor of $s$,
where standard GMRES is essentially $s$-step GMRES with the step size of one.
For instance, to apply SpMV $s$ times (Lines 7 to 9 of the pseudocode),
several CA variants of MPK exist~\cite{Mohiyuddin:2009}.
On a distributed-memory computer, 
CA variants may reduce the communication latency cost of SpMV, 
associated with the point-to-point neighborhood Halo exchange of the input vector,
by a factor of~$s$ (though it requires additional memory and local computation,
and it may also increase the total communication volume).
However, while in practice, SpMV is typically combined with a preconditioner to accelerate the convergence rate of GMRES, only a few CA preconditioners of specific types have been proposed \cite{Grigori:2015,Yamazaki:2014}. 
To support a wide-range of preconditioners used by applications, instead of CA MPK,
Trilinos $s$-step GMRES uses a standard MPK (applying each SpMV with neighborhood communication in sequence),
and focuses on improving the performance of block orthogonalization.
Also, avoiding the global communication may lead to a greater gain on the orthogonalization performance than CA MPK does on SpMV performance. 

In this paper, to maitain the stability of $s$-step GMRES, we focus on block orthogonalization schemes that can maintain the overall $\mathcal{O}(\epsilon)$ orthogonality error of the generated orthonormal block basis vectors $Q_{1:j}$, where $\epsilon$ is the machine precision:
\begin{eqnarray}\label{eq:eps}
  \|I - Q_{1:j}^T Q_{1:j}\| & = & \mathcal{O}(\epsilon).
\end{eqnarray}

The block orthogonalization algorithm consists of two steps:
the inter- and intra-block orthogonalization to orthogonalize the new set of $s+1$ basis vectors
against the previous vectors and among themselves, respectively.
To maintain orthogonality, in practice, both steps are applied
with re-orthogonalization. 

There are several combinations of the inter- and intra-block orthogonalization algorithms~\cite{Carson:2021}, but in this paper, we focus on 
the block-orthogonalization process that uses
Block Classical Gram-Schmidt (BCGS)
both for the first inter-block orthogonalization and for the re-orthogonalization,
and uses Cholesky QR (CholQR) factorization~\cite{Stath:2002}  for the intra-block \linebreak re-orthogonalization.
To ensure the stability and performance of the overall block orthgonalization, the remaining critical component is
the first intra-block orthogonalization scheme~\cite{Barlow:2024}, which is the focus of this paper.
Beside this first intra-block orthogonalization, such a block orthogonalization can be implemented using mostly BLAS-3 operations and needs only three global reduces. As a result, it performs well on current hardware architectures.
The pseudocode of this block orthogonalization process is shown in Figure \ref{algo:bcgs2}.

\begin{figure}[t]
\begin{center}
  \begin{subfigure}[b]{.9\linewidth}
  \centerline{\fbox{\begin{minipage}[h!]{\textwidth}
    \footnotesize
    \input{codes/bcgs}
  \end{minipage}}}
  \caption{BCGS Inter-block orthogonalization.} \label{algo:bcgs}
  \end{subfigure}
  \begin{subfigure}[b]{.9\linewidth}
  \centerline{\fbox{\begin{minipage}[h!]{\textwidth}
    \footnotesize
    \input{codes/cholQR}
  \end{minipage}}}
  \caption{CholQR Intra-block orthogonalization.} \label{algo:cholqr}
  \end{subfigure}
  \begin{subfigure}[b]{.9\linewidth}
  \centerline{\fbox{\begin{minipage}[h!]{\textwidth}
    \footnotesize
    \input{codes/bcgs2}
  \end{minipage}}}
  \caption{BCGS twice (BCGS2).} \label{algo:bcgs2}
  \end{subfigure}
  \caption{Block Classical Gram-Schmidt (BCGS) to orthogonalize $V_j$ against
           the orthonormal vectors $Q_{1:j-1}$, where
           ``chol($G$)'' returns the upper-triangular Cholesky factor of $G$.}
\end{center}
\end{figure}

In~\cite{Barlow:2024}, it has been shown that
in order to ensure the $\mathcal{O}(\epsilon)$ orthogonality error of all the basis vectors $Q_{1:j}$, the first intra-orthogonalization algorithm (Line 3 of Figure \ref{algo:bcgs2}) needs to generate $\widehat{Q}_j$ such that the backward and orthogonality errors satisfy
\begin{eqnarray}\label{eq:intra-blk-0}
  \|\widehat{V}_j - \widehat{Q}_j \widehat{R}_{j,j}\|~/~\| \widehat{V}_j \| & = & \mathcal{O}(\epsilon)
\end{eqnarray}
and
\begin{eqnarray}\label{eq:intra-blk}
  \|I - \widehat{Q}_j^T \widehat{Q}_j\| & = & \mathcal{O}(\epsilon).
\end{eqnarray}
In the next section, we explore two algorithms, which performs well on the current hardware architectures and achieves both \eqref{eq:intra-blk-0} and \eqref{eq:intra-blk}, for the first intra-block orthogonalization. 
Although we can show that the condition \eqref{eq:intra-blk-0} holds for the algorithms,
we focus on discussing the condition on each block vector (i.e., the required condition number $\kappa(\widehat{V}_j)$ of the input basis vectors $\widehat{V}_j$)
that is sufficient in order for each algorithm to guarantee~\eqref{eq:intra-blk} on the orthogonality error.

%

\ignore{
\begin{figure}[t]
\begin{center}
  \begin{subfigure}[b]{.9\linewidth}
  \centerline{\fbox{\begin{minipage}[h!]{\textwidth}
    \footnotesize
    \input{codes/BCGS-pip}
  \end{minipage}}}
  \caption{BCGS with Pythagorean Inner Product (BCGS-PIP). } \label{algo:bcgs-pip}
  \end{subfigure} 
  \begin{subfigure}[b]{\linewidth}
  \centerline{\fbox{\begin{minipage}[h!]{.9\linewidth}
    \footnotesize
    \input{codes/BCGS2-pip}
  \end{minipage}}}
  \caption{BCGS-PIP twice (BCGS-PIP2).}
  \end{subfigure} 
\end{center}
  \caption{ BCGS with Pythagorean Inner Product to generate a new set of orthonormal basis vectors $Q_j$. } \label{algo:bcgs2-pip}
\end{figure}
}

\section{Intra-Block Orthogonalization Algorithms}
\label{sec:one-stage}

\subsection{CholQR twice (CholQR2)}\label{sec:cholqr2}

For the first intra-block orthogonalization,
we first explore the use of CholQR~\cite{Stath:2002} twice (CholQR2). As we can see in Figure~\ref{algo:cholqr},
CholQR can be implemented mostly based on BLAS-3 and requires only one synchronization.

In~\cite[Theorem IV.1]{Yamazaki:2024,Stath:2002,CholeskyQR2ErrAnalysis},
it has been shown that the orthogonality error of $\widehat{Q}_j$ generated by CholQR in Figure~\ref{algo:cholqr} is bounded as 
\begin{equation}\label{eq:cholqr}
  \|I-\widehat{Q}_j^T\widehat{Q}_j\| = \mathcal{O}(\epsilon)\kappa(\widehat{V}_{j})^2,
\end{equation}
when the following condition is satisfied:
\begin{equation}\label{eq:assumption-1}
  c_1(\epsilon, n,s)\kappa(\widehat{V}_j)^2 < 1/2,
\end{equation}
where $c_1(\epsilon, n,s)$ is a constant.
Hence, BCGS2 with CholQR2 obtains the overall $\mathcal{O}(\epsilon)$ orthogonality error as formalized in the following proposition.
\begin{proposition} \label{prop1}
When the condition~\eqref{eq:assumption-1} is satisfied,
the orthgonality error 
of the basis vectors~$\widehat{Q}_j$ computed by the CholQR2-based first intra-block orthogonalization,
and hence of the basis vectors $Q_{1:j}$ generated by BCGS2 with CholQR2,
is of the order of the machine precision 
(as the condition \eqref{eq:intra-blk} ensures the condition \eqref{eq:eps}).
\end{proposition}
\begin{proof}
The proof of Proposition \ref{prop1} follows by combining the results from two facts: first, BCGS2 attains $O(\epsilon)$ orthogonality error overall provided each ``IntraBlk" step on Line 3 of Figure~\ref{algo:bcgs2} satisfies \eqref{eq:intra-blk-0} and \eqref{eq:intra-blk} \cite{Barlow:2024}, and CholQR2 satisfies these conditions when \eqref{eq:assumption-1} is satisfied \cite[Theorems 3.3 \& 3.5]{Fukaya:2015}.
\end{proof}

The main drawback of CholQR is that it
computes the Gram matrix of the input basis vectors $\widehat{V}_{j}$ to be orthogonalized (Line~2 in Figure~\ref{algo:cholqr}),
and the Gram matrix has the condition number which is the square of the condition number of the input vectors $\widehat{V}_j$. 
Hence, CholQR can fail when
the condition number of the input vectors $\widehat{V}_{j}$ is greater than
the reciprocal of the square-root of the machine epsilon $\epsilon$
(i.e., $\kappa(\widehat{V}_{j}) > 1/\mathcal{O}(\epsilon^{1/2})$). 
This can cause numerical issues, especially for the $s$-step method, because
even when Newton or Chebyshev basis is generated,
the $s$-step basis vectors can be ill-conditioned with a large condition number.

\begin{figure}
\begin{center}
  \centerline{\fbox{\begin{minipage}[h!]{.9\linewidth}
    \footnotesize
    \input{codes/cholQR_recursive}
  \end{minipage}}}
 \end{center}
  \caption{Recursive Cholesky QR (CholQR) to orthonormalize a set of vectors $V \in \mathbb{R}^{n\times s}$, where ``$\mbox{chol}(G)$ returns the upper-triangular Cholesky factor of the Gram matrix $G$.}\label{algo:cholqr-recursive}
\end{figure}

To alleviate this potential numerical instability,
Trilinos implements a ``recursive'' variant of CholQR, as shown in Figure~\ref{algo:cholqr-recursive};
when Cholesky factorization of the Gram matrix fails
at the $k$th step due to a non-positive diagonal,
it orthogonalizes just the first $k-1$ vectors by CholQR.
It then orthogonalizes the remaining vectors against the first $k-1$ (roughly) orthonormal vectors by BCGS
and recursively calls CholQR on the remaining vectors.
This avoids the algorithmic breakdown of the orthogonalization process
by adaptively adjusting the block size to orthogonalize the $s+1$ basis vectors.
To orthogonalize the remaining vectors against the first $k-1$ vectors,
it uses the partial Cholesky factors, and hence no additional overhead is needed.
However, it needs to re-compute the dot-products of the remaining vectors,
and may require multiple global reduces for ill-conditioned basis vectors. Furthermore, though it is often effective in recovering from the failures in combination with MPK, there is no bound on the orthogonality error with this recursive variant.

\subsection{Randomized-Householder CholQR (RandCholQR)}
\label{sec:randCholQR}

Since MPK can generate ill-conditioned basis vectors, the requirement~\eqref{eq:assumption-1} can be too restrictive. To enhance the numerical robustness, we integrate the random-sketching techniques.

\begin{figure}[t]
\begin{center}
  \begin{subfigure}[b]{.9\linewidth}
  \centerline{\fbox{\begin{minipage}[h!]{.9\linewidth}
    \footnotesize
    \input{codes/randQR_nochol}
  \end{minipage}}}
  \caption{Randomized Householder QR (RandHH) } \label{algo:randQR_nochol}  
  \end{subfigure}
\end{center}

\begin{center}
  \begin{subfigure}[b]{.9\linewidth}
  \centerline{\fbox{\begin{minipage}[h!]{.9\linewidth}
    \footnotesize
    \input{codes/randQR}
  \end{minipage}}}
  \caption{Randomized Householder CholQR (RandCholQR)} \label{algo:randQR}
  \end{subfigure}
\end{center}
  \caption{
  Randomized QR algorithm, where $\mbox{HH}(V)$ returns the orthogonal basis vectors $Q$ and the upper-triangular matrix $R$ based on the Householder QR algorithm such that $V=QR$.} 
\end{figure}


Instead of forming the Gram matrix of the basis vectors, which is the main cause of the numerical instability, Randomized CholQR (RandCholQR) first computes the random sketch $\widehat{\mathbf{v}}_j$ of the basis vectors~$\widehat{V}_j$. It then generates their well-conditioned basis vectors~$\overline{Q}_j$ using the upper-triangular matrix computed by the stable QR factorization of the sketched vectors $\widehat{\mathbf{v}}_j$. When the input vectors~$\widehat{V}_j$ are numerically full-rank
and the sketch matrix $\Theta$ is a $(\mu, \delta, s)$-oblivious $\ell_2$-subspace embedding (see Definition~\ref{def:obliviousSubspaceEmbedding}), it can be shown that the generated basis vectors $\overline{Q}_j$ has the condition number of $\mathcal{O}(1)$~\cite{Balabanov:2021:block,Higgins:2025}. Hence, according to~\eqref{eq:cholqr}, as we call CholQR on $\overline{Q}_j$, the resulting vector $\widehat{Q}_j$ has an $\mathcal{O}(\epsilon)$ orthogonality error.  Figure~\ref{algo:randQR} shows the pseudocode of the resulting RandCholQR algorithm for the intra-block orthogonalization.

The sketching typically requires one global all-reduce, and as a result, we expect RandCholQR to perform similarly to CholQR2.
The actual performance of the algorithm depends on the type of the sketching being used. In this paper, we look at the following random-sketching techniques that can be implemented using standard linear algebra kernels. Regardless of the types of the sketching used, the last step of the randCholQR, to generate the well-conditioned basis vectors through forward substitutions, requires $\mathcal{O}(ns^2)$ computation (Line 6 of Figure~\ref{algo:randQR_nochol}). Hence, in the discussion below, we focus on the first two steps of the algorithm (Lines 2 and 4).
\begin{itemize}
\item Gaussian-Sketching can be implemented using a dense GEneral Matrix-Matrix multiply (GEMM) to compute $\widehat{\mathbf{v}}_{j}$ on Line 2 of Figure~\ref{algo:randQR_nochol}. The nice feature of this approach is that it requires the sketch size of only $\mathcal{O}(s)$ \cite{Martinsson:2020}. However, this is a ``dense'' sketch, and the dense sketch matrix $\Theta$ needs to be explicitly stored to call GEMM. Hence, the Gaussia Sketch has the storage overhead of $\mathcal{O}(ns)$ and the computational complexity of $\mathcal{O}(ns^2)$ to generate the sketch $\widehat{\mathbf{v}}_j$ (Line 2), and its overall computational complexity is $\mathcal{O}(ns^2 + s^3)$, where $\mathcal{O}(s^3)$ is for computing HH of $\widehat{\mathbf{v}}_j$ (Line 4).

Though this complexity is the same as CholQR, the complexity of RandHH with the dense Gaussian-Sketch has a larger constant associated with the sketch size (i.e., $ns^2$ flops for CholQR to compute the Gram matrix, compared to $4ns^2$ floating-point operations (flops) for RandHH to generate the sketched vectors with the sketch size of $2s$).

The terms associated with $s^2$ and $s^3$ in the complexity $\mathcal{O}(ns^2 + s^3)$ could become significant, especially when we need to sketch a large number of basis vectors (e.g., for the two-stage approach discussed in Section~\ref{sec:two-stage}).

\item Count-Sketching can be implemented using Sparse-Matrix Matrix (dense vectors) multiply (SpMM) with a sparse sketch matrix $\Theta$ having one nonzero entry in each row  (with numerical values of either $1$ or $-1$). This is a ``sparse'' sketching, and compared to Gaussian-Sketching, it has a lower storage cost of $\mathcal{O}(n)$ and a lower computational complexity of $\mathcal{O}(ns)$. 

One drawback, however, is that it requires the larger sketch size of $\mathcal{O}(s^2)$ \cite{Clarkson:2013}. This could be a significant performance overhead, or a sequential performance bottleneck, when we need to sketch a large number of basis vectors,~$s$. In particular, on a distributed-memory computer, the basis vectors $\widehat{V}_j$ are distributed among the MPI processes in a 1D block row format (see Section~\ref{sec:imple} for more detailed discussion about our implementation).
Hence to form sketched vectors~$\widehat{\mathbf{v}}_j$, it requires the global all-reduce of $\mathcal{O}(s^2)$-by-$s$ dense sketched vectors, i.e., 
\[
  \widehat{\mathbf{v}}_j := \sum_{p=1}^{n_p} (\Theta^{(p)})^T \widehat{V}_j^{(p)},
\] 
where $\Theta^{(p)}$ and $\widehat{V}_j^{(p)}$ are the parts of $\Theta$ and $\widehat{V}_j$, which is distributed on the $p$-th process, respectively, and $\sum_{p=1}^{n_p}$ is the global all-reduce. Though both Gaussian and Count sketching requires all-reduce, the communication volume $\mathcal{O}(s^3)$, and hence the time, needed for the all-reduce with Count sketching can become significantly more, especially for a large $s$, compared to $\mathcal{O}(s^2)$ for Gaussian sketching (e.g., Figure~\ref{fig:distributed-intra-time}, and for the two-stage algorithm).

In addition, the Householder QR factorization of the sketched vectors is performed redundantly on a CPU by each MPI process (Line 4 of Figure~\ref{algo:randQR_nochol}), leading to a potential parallel performance bottleneck (i.e., the $\mathcal{O}(s^4)$ complexity for the local computation with the Count-Sketch, compared to the $\mathcal{O}(s^3)$ complexity with the Gaussian-Sketch).

The overall computational complexity of the Count Sketch is $\mathcal{O}(ns + s^4)$, compared to $\mathcal{O}(ns^2 + s^3)$ of the Gaussian sketch. 

\item To combine the advantage of the above two sketching approaches, Count-Gaussian Sketching uses a $\mathcal{O}(s^2) \times n$ Count-Sketch followed by a $\mathcal{O}(s) \times \mathcal{O}(s^2)$ Gaussian Sketch. Hence, most of the computational and storage costs are due to Count-Sketch, and then the Gaussian Sketch is locally applied such that the size of the final sketched vector is $\mathcal{O}(s)$.

Compared to the Count-Sketch, the size of the global all-reduce is reduced from  $\mathcal{O}(s^2)$-by-$s$ to $\mathcal{O}(s)$-by-$s$, i.e., 
\[
  \widehat{\mathbf{v}}_j := \sum_{p=1}^{n_p} \Theta_g^T((\Theta_c^{(p)})^T \widehat{V}_j^{(p)}),
\]
where $\Theta_c$ and $\Theta_g$ are the Count and Gaussian sketch matrices, respectively. The communication volume is reduced from the Count-Sketch because the Gaussian sketch is applied locally before the global all-reduce. Hence the communication volume for the all-reduce is $\mathcal{O}(s^2)$, and is the same as the Gaussian-Sketch and is reduced from $\mathcal{O}(s^3)$ needed for Count-Sketch.

Though it now requires only $\mathcal{O}(s^3)$ computation to compute the QR factorization of the final sketch,
the Gaussian sketch is applied redundantly on each MPI process, which has the computational complexity of $\mathcal{O}(s^4)$ and is the same as that needed for the Count-Sketch. Nevertheless, the local GEMM for the Count-Gauss may perform more efficiently than the local HH for Count-Sketching. Moreover, our implementation gathers the results $\widehat{\mathbf{v}}_j$ of the all-reduce on CPU, and GEMM before the global-reduce is computed on a GPU, while HH after the global reduce is computed on CPU.

Compared to the Gaussian-Sketch (or to the generation of the Gram matrix for CholQR), the Count-Gaussian reduces the computation complexity to apply the sketch from $\mathcal{O}(ns^2)$ to $\mathcal{O}(ns)$. Moreover, even if the Count-Gaussian did not lead to a performance improvement over the Gaussian-Sketch in practice
(due to a more efficient dense matrix-matrix multiply implementation, compared to a sparse-matrix matrix multiply),
the storage requirement is reduced from $\mathcal{O}(ns)$ to $\mathcal{O}(n)$.

The overall computational complexity of the Count-Gaussian is $\mathcal{O}(ns + s^4)$.
\end{itemize}

\begin{proposition}\label{prop2}
When RandCholQR is used as the first intra-block orthogonalization, the orthogonality error of the resulting basis vectors $\widehat{Q}_{j}$ is  of the order of the machine precision when $\Theta^T$ forms an $\mu$-subspace embedding over the range of $\widehat{V}_{j}$, and the following condition is satisfied:
\begin{equation}\label{eq:assumption-2}
  c_2(\epsilon, n,s)\kappa(\widehat{V}_j) < 1/2,
\end{equation}
where $c_2(\epsilon, n,s)$ is a constant,
or equivalently if MPK generates numerically full-rank basis vectors $\widehat{V}_j$.
\end{proposition}
\begin{proof}
The proof of Proposition \ref{prop2} follows by combining the results from two facts: first, BCGS2 attains $O(\epsilon)$ orthogonality error overall provided each ``IntraBlk" step in line 3 satisfies \eqref{eq:intra-blk-0} and \eqref{eq:intra-blk} \cite{Barlow:2024}, and RandCholQR satisfies these conditions when \eqref{eq:assumption-2} is satisfied \cite[Theorem 5.3]{Higgins:2025}.
\end{proof}

Hence, compared to CholQR2, which requires~\eqref{eq:assumption-1}, RandCholQR improves the numerical robustness of the overall block orthogonalization process, while we expect only a small overhead in term of its execution time (as shown in Section~\ref{sec:perf}).

\section{Two-stage Block Orthogonalization Framework}
\label{sec:two-stage}

\begin{figure}[t]
  \begin{center}
  \fbox{\begin{minipage}[h!]{.9\linewidth}
    \footnotesize
    \input{codes/2stage}
  \end{minipage}}
  \end{center}
  \caption{Two-stage Block Orthgonalization with MPK.}
  \label{algo:two-stage}
\end{figure}
\begin{figure}[t]
\begin{center}
  \fbox{\begin{minipage}[h!]{.9\linewidth}
    \footnotesize
    \input{codes/2stage-flat}
  \end{minipage}}
  \caption{Two-stage Block Orthgonalization without MPK.}
  \label{algo:two-stage-flat}
\end{center}
\end{figure}

BCGS2 discussed in Section~\ref{sec:one-stage} can orthogonalize the set of $s+1$
basis vectors with only five synchronizations
and using BLAS-3 operations to performs most of its local computation. 
However, its performance may be still limited
by the small step size $s$ required to maintain the stability of MPK (e.g., to ensure that the $s+1$ basis vectors are numerically full rank).
In order to improve the performance of the block orthogonalization while using the small step size, ``two-stage'' block orthogonalization algorithms have been proposed~\cite{Yamazaki:2024}.
Instead of fully-orthogonalizing the basis vectors at every $s$ steps,
the two-stage algorithm only ``pre-processes'' the $s$-step basis vectors $V_j$ at every $s$ steps. The objective of this first stage is to maintain the well-conditioning of the basis vectors at a cost that is lower than that required for the full orthogonalization or with the same cost as the initial orthogonalization (e.g., BCGS with CholQR).
Then once a sufficient number of basis vectors, $\widehat{s}$, are generated to obtain higher performance,
the second stage orthogonalizes the $\widehat{s}$ basis vectors at once.
For the following discussion,
we refer to the blocks of $s$ and $\widehat{s}$ vectors as the ``panels'' and ``big panels'', respectively.
Also, to distinguish from the two-stage algorithms,
we refer to the BCGS algorithms discussed in Sections~\ref{sec:sstep} and \ref{sec:one-stage} as ``one-stage'' algorithms.

For this paper, we focus on the specific variant of the two-stage algorithm shown in Figure~\ref{algo:two-stage}. This variant first (roughly) orthogonalizes the new panel of basis vectors $V_j$ against the previous big panels using BCGS (Line 10). It then pre-processes the resulting panel vectors $\widehat{V}_j$ within the current big panel (Line 12). Finally, the second stage fully orthogonalizes the big panel, first orthogonalizing the big panel using CholQR (Line 15), followed by BCGS with CholQR on the big panels (Line 19). 

The main objective of the pre-processing stage is to keep the condition number of the basis vectors small at a low cost. Namely, for the two-stage algorithm to maintain the same level of the stability as the one-stage algorithm, after the first inter-big panel BCGS (Line 10 of Figure~\ref{algo:two-stage}), the condition number of the big panel is hoped to be the same order as that of each panel in the one-stage algorithm,
\begin{equation}
\label{eq:two-stage-preprocess}
   \kappa(\widehat{V}_{\ell+1:j}) \approx \kappa(\widehat{V}_{\ell+1}).
\end{equation}
Without the pre-processing step (i.e., one-stage with $s=\widehat{s}$), the condition number of the basis vectors will increase exponentially with the step size (see numerical results in~\cite{Yamazaki:2024}).

When it is seen as the generic block orthogonalization scheme, without MPK, in Figure~\ref{algo:two-stage-flat}, this two-stage algorithm can be considered as BCGS2 (shown in Figure~\ref{algo:bcgs2}) that orthogonalizes the big panel as a set of block vectors, where the pre-processing step, followed by CholQR,
is used as the first intra-block orthogonalization (Lines 3 to 8). 
Hence, if the condition number of the big panel $\widehat{V}_{\ell+1:t+1}$, after the pre-processing step (Lines 4 to 6) is $\mathcal{O}(1)$, then the orthogonality error of the big panel $\widehat{Q}_{\ell+1:t+1}$ after the first CholQR (Line 7) is $\mathcal{O}(\epsilon)$, and thus, the overall two-stage block orthogonalization is stable with $\mathcal{O}(\epsilon)$ orthogonality error of the resulting vectors, i.e.,
$$\|I - Q_{1:t+1}^T Q_{1:t+1}\| = \mathcal{O}(\epsilon).$$

In the next section, we introduce two pre-processing schemes and
discuss the condition on the input big panel $\kappa(\widehat{V}_{\ell+1:t+1})$ to ensure 
the overall stability of the two-stage algorithm. 


\section{Preprocessing Schemes for Two-stage Framework}
\label{sec:preproc}

We first formally establish the following proposition.
\begin{proposition}\label{prop_2stage}
If a pre-processing scheme can maintain the condition number of the big panel to be bounded as
\begin{equation}\label{eq:assumption}
    \kappa\left(\overline{Q}_{\ell+1:t+1}\right) = O(1),
\end{equation}
then the overall stability of the two-stage scheme \eqref{eq:eps} is ensured.
\end{proposition}
\begin{proof}
Using~\eqref{eq:cholqr}, condition \eqref{eq:assumption} implies that $\widehat{Q}_{\ell+1:t+1}$ produced by 
the first intra big-panel CholQR factorization of $\overline{Q}_{\ell+1:t+1}$ 
at Line 7 of Figure \ref{algo:two-stage-flat} satisfies
\begin{equation*}
    \|I - \widehat{Q}_{\ell+1:t+1}^T\widehat{Q}_{\ell+1:t+1}\| = O(\epsilon).
\end{equation*}
Therefore, the first intra-block orthogonalization, which is based on the pre-processing scheme satisfying \eqref{eq:assumption} followed by CholQR, satisfies the orthogonality error \eqref{eq:intra-blk-0} of the big-panel,
which is required by \cite{Barlow:2024} to ensure the overall $\mathcal{O}(\epsilon)$ orthogonality error~\eqref{eq:eps} of the two-stage scheme.
\end{proof}

\subsection{BCGS with Pythagorean Inner Product (BCGS-PIP)}

\begin{figure}[t]
  \centerline{\fbox{\begin{minipage}{.9\linewidth}
    \footnotesize
    \input{codes/BCGS-pip}
  \end{minipage}}}
  \caption{BCGS with Pythagorean Inner Product (BCGS-PIP).}
  \label{algo:bcgs-pip}
\end{figure}

Our first pre-processing scheme for the two-stage algorithm
is based on the single-reduce variant~\cite{Carson:2021,Yamazaki:2020} of BCGS with CholQR,
shown in Figure~\ref{algo:bcgs-pip}.
Instead of explicitly computing the Gram matrix, this variant uses the Pythagorean rule, hence requiring only one global-reduce for both the inter-block and intra-block orthogonalization.\footnote{
For the first block (i.e., $j = \ell + 1$), BCGS-PIP is equivalent to CholQR.}
In addition, it was shown~\cite[Theorem 3.4]{Carson:2021} that 
if the condition number of the input basis vectors is bounded as
\begin{equation}
  \label{eq:assumption-3}
  c_3(\epsilon, n, \widehat{s})\kappa(\widehat{V}_{\ell+1:t+1})^2 < 1/2,
\end{equation}
where $c_3(\epsilon, n,s)$ is a constant and
$t+1$ is the last block index of the big panel (i.e., $t := \ell + \widehat{s}/s - 1$), then
the orthogonality error of the big panel computed by BCGS-PIP satisfies
\begin{equation}
 \|I - \overline{Q}_{\ell+1:t+1}^T \overline{Q}_{\ell+1:t+1}\| \le c_3(\epsilon, n, \widehat{s})\kappa(\widehat{V}_{\ell+1:t+1})^2.\label{eq:orthBoundQhatQ}
\end{equation}

Hence, we have the following proposition:
\begin{proposition}\label{prop71}
If this pre-processing scheme can maintain the condition number of the big panel to satisfy the condition~\eqref{eq:assumption-3},
then the overall stability of the two-stage scheme is ensured.
\end{proposition}
\begin{proof}
Because of the assumption~\eqref{eq:assumption-3} and the bound~\eqref{eq:orthBoundQhatQ}, we have
\begin{equation*}
    \| I - \overline{Q}^T_{\ell+1:t+1}\overline{Q}_{\ell+1:t+1}\| < 1/2. 
\end{equation*}
Thus, by Weyl's inequality, we obtain
\[
\left\{
\begin{array}{l} 
    \sigma_{\min} \left(\overline{Q}^T_{\ell+1:t+1}\overline{Q}_{\ell+1:t+1}\right) \geq 1 - \| I - \overline{Q}^T_{\ell+1:t+1}\overline{Q}_{\ell+1:t+1}\| > 1/2\\
    \sigma_{\max} \left(\overline{Q}^T_{\ell+1:t+1}\overline{Q}_{\ell+1:t+1}\right) \leq 1 + \| I - \widehat{Q}^T_{\ell+1:t+1}\overline{Q}_{\ell+1:t+1}\| < 3/2.
\end{array}
\right.
\]
Therefore,
\begin{equation}
    \kappa\left(\overline{Q}_{\ell+1:t+1}\right) = \sqrt{\kappa\left(\overline{Q}_{\ell+1:t+1}^T\overline{Q}_{\ell+1:t+1}\right)} = O(1), \label{eq:Qhatcond}
\end{equation}
and the proof follows by Proposition~\ref{prop_2stage}.
\end{proof}

A similar two-stage scheme based on BCGS-PIP was studied in~\cite{Yamazaki:2024}. The variant studied in this paper is slightly different and has a more stable behavior. Namely, its robustness depends on the condition number of the big panel $\widehat{V}_{\ell+1:t+1}$ as shown in the condition~\eqref{eq:assumption-3}, while the stability of the previous variant depended on the condition number of the accumulated big panels, i.e., $c(\epsilon) \kappa([Q_{1:\ell},\widehat{V}_{\ell+1:t+1}])^2 < 1/2$.

Similarly to CholQR, BCGS-PIP can fail when
the condition number of the big panel is greater than
the reciprocal of the square-root of the machine epsilon $\epsilon$. 
This can cause numerical issues, especially for the ill-conditioned basis vectors generated by MPK.

\subsection{Randomized BCGS}

\begin{figure}[t]
  \begin{subfigure}[b]{\linewidth}
  \centerline{\fbox{\begin{minipage}{.9\linewidth}
    \footnotesize
    \input{codes/incRand}
  \end{minipage}}}
  \caption{The $j$th step (for PreProc for Figure~\ref{algo:two-stage-flat})} \label{algo:incRand1}
  \end{subfigure}
\smallskip
  \begin{subfigure}[b]{\linewidth}
  \centerline{\fbox{\begin{minipage}{.9\linewidth}
    \footnotesize
    \input{codes/incRand_v2}
  \end{minipage}}}
  \caption{Accumulated steps.} \label{algo:incRand2}
  \end{subfigure}  
  \caption{Randomized BCGS2 (RandBCGS2).}
  \label{algo:incRand}
\end{figure}

To enhance the stability of the pre-processing scheme based on BCGS-PIP, we consider a randomized BCGS scheme shown in Figure~\ref{algo:incRand} as our pre-processing algorithm. This algorithm sketches the big panel, but $s$ basis vectors at a time. To ensure stability, we orthogonalize the sketched vectors $\widehat{\mathbf{v}}_{\ell+1:t+1}$ using BCGS2 with Householder intra-block orthogonalization.
Similar randomized block orthogonalization algorithms were discussed in~\cite{Balabanov:2021:block,Balabanov:2022:GMRES,Higgins:2025}. We used this randomized algorithm as the pre-processing scheme for the two-stage BCGS2, in order to maintain the well-conditioning of the big panel $\overline{Q}_{\ell+1:t+1}$ and obtain the overall $\mathcal{O}(\epsilon)$ orthogonality error of $Q_{1:t+1}$.

\begin{proposition}\label{prop72}
The overall stability of two-stage algorithm is ensured when randomized BCGS pre-processing is used with $\Theta^T$ that forms an $\mu$-subspace embedding over the range of $\widehat{V}_{\ell+1:t+1}$, and the following condition is satisfied:
\begin{equation}\label{eq:assumption-4}
  c_4(\epsilon, n, \widehat{s})\kappa(\widehat{V}_{\ell+1:t+1}) < 1/2.
\end{equation}
where $c_4(\epsilon, n,s)$ is a constant,
or equivalently when MPK generates numerically full-rank basis vectors for each big panel $\widehat{V}_{\ell+1:t+1}$.
\end{proposition}
\begin{proof}
Observe that RandBCGS2 (in Figure \ref{algo:incRand2}) is identical to RandHH (in Figure \ref{algo:randQR_nochol}) except that BCGS2-HH is used, instead of HH factorization, on the sketched vectors to generate the upper-triangular matrix. In \cite[Corollary 5.2]{Higgins:2025}, it was proven that RandHH results in $\kappa\left( \overline{Q}_{\ell+1:t+1}\right) = O(1)$. 

We prove in Appendix that the backward error of the QR factorization via BCGS2-HH only differs from the Householder QR backward error by a constant factor. Hence, the backward error analysis in \cite[Section 5.2.4]{Higgins:2025} only differs by a constant factor when BCGS2+HH, instead of HH, was used on the sketched vectors. 
Therefore, by \cite[Corollary 5.2]{Higgins:2025}, RandBCGS2 will generate the basis vectors $\overline{Q}_{\ell+1:t+1}$ whose condition number is bounded by $\kappa\left( \overline{Q}_{\ell+1:t+1}\right) = O(1)$, and the proof follows by Proposition~\ref{prop_2stage}.
\end{proof}
Assuming that the condition number of the big panel is in the same order as that of each panel in the one-stage algorithm (as in the condition~\eqref{eq:two-stage-preprocess}),
the condition~\eqref{eq:assumption-4} on the big panel for the two-stage algorithm is equivalent to the condition~\eqref{eq:assumption-2} on the panel for the one-stage algorithm, and hence the two-stage algorithm with RandBCGS2 is as stable as 
the one-stage algorithm with RandHH.

The roundoff error analysis of a randomized BCGS, that is similar to the one in Figure~\ref{algo:incRand}, has been presented in~\cite{Balabanov:2021:block}, where
their framework allows generating the sketches of the block for the inter and intra block orthogonalization, separately, while in this paper, we focus on the one in Figure~\ref{algo:incRand} that uses a single sketch of the big panel (similar to randCholQR, but by generating the sketch of each block at a time).

BCGS-PIP and RandBCGS2, shown in Figures~\ref{algo:bcgs-pip}
and \ref{algo:incRand}, respectively, have a similar algorithmic structure. In particular, at every $s$ steps, both algorithms would require one global-reduce, followed by a small local computation (either Cholesky factorization of the Gram matrix or BCGS2 orthogonalization of the sketched vectors) and forward-substitution to generate the basis vectors $\overline{Q}_j$. Though RandBCGS2 has the larger complexity for the local computation, the main factor that impacts their performance difference is the dot-products required for BCGS-PIP and the random-sketching required for RandGCGS.

\begin{table*}[t]
\centerline{\scriptsize
\begin{tabular}{l||ll|ll|ll}
                      & \multicolumn{6}{c}{Flop Count}\\
                      & \multicolumn{2}{c|}{Projection} & \multicolumn{2}{c|}{Normalization} & \multicolumn{2}{c}{Total}\\
\hline\hline
%
%
One-stage:\\
\;\;CGS2                 & $2nm(m+1/2)$                 & $2nm(m+1/2)$           & $3nm$                        &   $3nm$                  & $2nm^2$         & $2nm^2$        \\ 
\;\;BCGS2+CholQR2                  & $2nm(s+1)(m-s+1/2)/s$        & $2nm(s+1)(m-s+1/2)/s$  & $4nm(s+1)(s+3/2)/s$        & $2nm(s+1)(s+3/2)/s$    & $2nm^2(s+1)/s$  & $2nm^2(s+1)/s$ \\ 
\;\;BCGS2+RCholQR($\hat{s}=s$)       & $2nm(s+1)(m-s+1/2)/s$        & $2nm(s+1)(m-s+1/2)/s$  & $4nm(s+1)(7s/4 + 2)/s$  & $2nm(s+1)(s+3/2)/s$    & $2nm^2(s+1)/s$  & $2nm^2(s+1)/s$ \\
\hline
Two-stage:\\
\;\;BCGS-PIP                      & $2nm(s+1)(m-s+1/2)/s$        & $2nm(s+1)(m-s+1/2)/s$  & $4nm(s+1)(s+3/2)/s$        & $2nm(s+1)(s+3/2)/s$    & $2nm^2(s+1)/s$  & $2nm^2(s+1)/s$ \\ 
\;\;RBCGS2($s<\hat{s}<m$)     &  $2nm(s+1)(m-\hat{s}+1/2)/s$ & $2nm(\hat{s}+1)(m-\hat{s}-1/2)/\hat{s}$  
                                                                                  & $2nm(s+1)(5\hat{s}/2 - s/2 + 5/2)/s$  & $2nm(\hat{s}+1)(\hat{s}+3/2)/\hat{s}$ 
                                                                                                                                            & $2nm^2(s+1)/s$ & $2nm^2(\hat{s}+1)/\hat{s}$ \\
                          &                              &                        & \;\;$ + 2nm(\hat{s}+1)(\hat{s}+3/2)/\hat{s}$  & &\\
\;\;RBCGS2($\hat{s}=m$)       & $2nm(s+1)(5m/2-s/2+5/2)/s$       &                        & $nm(s+1)^2/s$             & $2n(m + 1)(m + 3/2)$    & $5nm^2(s+1)/s$        & $2nm^2$     \\ 
\end{tabular}
}
\caption{Asymptotic computational complexity of orthogonalization scheme over one GMRES restart cycle,
         where ``RCholQR'' and ``RBCGS'' are RandCholQR and RandBCGS2, respectively.
         For the complexity with the random sketch, we use Gaussian Sketch with the sketch size of $\widehat{m}  = 2 (\widehat{s}+1)$.
         For each components, there are two columns for each component, ``Projection'', ``Normalization'', and ``Total'',
         where for the standard and $s$-step algorithms, the first and second columns are for orthogonalization and then for reorthogonalization,
         while for the remaining algorithms, they are for the preprocessing based on Gaussian random-sketching and then for orthogonalization.
         For two-stage, ``Projection'' and ``Normalization'' are the inter-block and intra-block orthogonalization of big panels.
         ``Total'' only shows the leading terms of the complexity.} \label{tab:costs-comp}
\end{table*}

\subsection{Remarks on Complexity}

We discuss the overall complexity of the various block orthogonalization algorithms in Section~\ref{sec:complexity}, but we provide a few remarks on the complexity, comparing the one-stage and two-stage algorithms, here.

Compared to the one-stage algorithm with RandHH, the two-stage algorithm with RandBCGS2 reduces the communication cost by delaying the orthogonalization until $\widehat{s}+1$ basis vectors are generated. However, this two-stage algorithm needs to sketch the big panel with the total of $\widehat{s}+1$ basis vectors, hence requiring a larger sketch size than the one-stage algorithm.
For instance, with the Gaussian sketch, its sketch size for the two-stage algorithm is proportional to the number of columns in the big panel, $\widehat{s}+1$, while the one-stage algorithm sketches each panel of $s+1$ basis vectors, and its sketch size is proportional to the panel size, $s+1$. Hence, in the randomized two-stage algorithm, there is a trade-off between the reduced communication cost for the orthogonalization and the increased computational cost for sketching, which we discuss in the next section.\footnote{%
As Count-Gaussian sketching has $O(n)$ complexity, it has the potential to remove this overhead of the randomized two-stage algorithm.}

When $\widehat{s} < m$, the two-stage algorithm with the two pre-processing schemes in Section~\ref{sec:preproc} requires about the same computational cost as the one-stage algorithm. On the other hand, when $\widehat{s}=m$, the two-stage algorithm has a lower computational cost than the one-stage algorithm (see the asymptotic complexity discussion in Section~\ref{sec:complexity}).
This is because
if the preprocesing step can keep the well-conditioning of the basis vectors over the whole restart-cycle of $s$-step GMRES, then the reorthogonalization (Line 19) is not needed, reducing the total computational cost of the orthogonalization.
In this case, it is also possible to skip the $\ell_2$-orthogonalization (i.e., CholQR on Line 15), and solve the least-square problem in the sketched space. Nevertheless, in this paper, we will focus on generating the $\ell_2$-orthogonal basis vectors, but will show the breakdown of the orthogonalization time, including the CholQR time, in Section~\ref{sec:perf}.

\section{Asymptotic Complexity}
\label{sec:complexity}

Tables~\ref{tab:costs-comp} and \ref{tab:costs-comm} compare the computational, storage, and communication
complexities of the different block orthogonalization schemes, respectively. 
These complexities are total costs using one MPI process, and not for distributed-memory.
\begin{itemize}
\item The $s$-step GMRES includes the starting vector to the set of the vectors to be orthogonalize.
      For instance, with $s=1$, $s$-step orthogonalizes two vectors at each step. This leads to 
      about $2\times$ more flops for ``Projection'' for $s$-step GMRES, compared to the standard GMRES.
      In addition, CholQR of two vectors requires about $(10/3)\times$ more flops
      than computing a dot-product and scaling of a single vector.

\item The randomized block orthogonalization algorithm performs the preprocessing based on random-sketching to generate the well-conditioned basis vectors,
       followed by the $\ell_2$-orthogonalization of the basis vectors. 
       Hence, the randomized algorithm replaces the first orthogonalization process of the standard algorithm
       with the randomized transformation of the basis vectors, while
       the orthogonalization process is identical and
       performs the same number of flops as the re-orthogonalization process of the standard algorithm.

\item  With $\widehat{s}=s$ (one-stage), compared to CholQR2,
       RandCholQR has higher computation complexity and larger communication volume,
       but the relative overhead is small (about the order of $s/m$) in the total complexity. This leads to insignificant
       increase in the orthogonalization time (for improving the numerical stability) in our performance tests.

\item One-stage algorithm with CholQR2 and two-stage algorithm with BCGS-PIP have about the same computational complexity.
       
\item With $\widehat{s}=m$, the two-stage framework has the best communication latency cost, but random-sketching has the
      highest computational overhead, where the total computational cost of the orthogonalization
      is increased by a factor of $1.75\times$. Nevertheless, when the performance is limited by the latency,
      this computational overhead may not be significant.
      
      Though the storage cost is increased by a factor of three, which
      could limit the use of the dense Gaussian sketch, the overhead can be reduced by using sparse Count sketch.

\item These storage and computational overheads may be reduced using a smaller sketch size ($s < \widehat{s} < m)$.
      This reduces the overhead to be $\mathcal{O}(\widehat{s}/m)$, though the reduction in the latency
      is also reduced.
      Nevertheless, in our experiments, the two-stage algorithm obtained the best performance using $\widehat{s}=m$.
\end{itemize}

\begin{table}[t]
\begin{center}\footnotesize
\begin{tabular}{l||l|ll}
                      &                 & \multicolumn{2}{c}{Communication}\\
                      & Storage         & \multicolumn{1}{c}{Latency} & \multicolumn{1}{c}{Volume}\\
\hline                
CGS2                  & $nm$            & $4m$               & $nm(2m+4)$\\
BCGS2+CholQR2         & $nm$            & $4\frac{m}{s}$     & $nm\frac{2m+4+4s}{s}$\\
BCGS2+RCholQR($\widehat{s}=s$)   & $n(m+\widehat{m})$  & $4\frac{m}{s}$     & $nm\frac{2m+4+4s+\widehat{m}}{s}$\\
RBCGS2($s<\widehat{s}<m$) & $n(m+\widehat{m})$  & $\frac{m}{s} + 3\frac{m}{\widehat{s}}$ & $nm(\frac{m+2+2s+\widehat{m}}{s}+\frac{m+2+4\widehat{s}}{\widehat{s}})$\\
RBCGS2($\widehat{s}=m$)   & $n(m+\widehat{m})$  & $\frac{m}{s} + 1$          & $nm\frac{m/2+\widehat{m}+2+2s}{s} + 2nm$\\
\end{tabular}
\caption{Asymptotic storage and communication costs of orthogonalization scheme over one GMRES restart cycle,
         where ``RCholQR'' and ``RBCGS'' are RandCholQR and RandBCGS2, respectively.
} \label{tab:costs-comm}
\end{center}
\end{table}

\section{Numerical Experiments}
\label{sec:numeric}

\begin{figure}[t]
\centerline{
  \begin{subfigure}[b]{.53\linewidth}
   \centerline{
     \includegraphics[width=\linewidth]{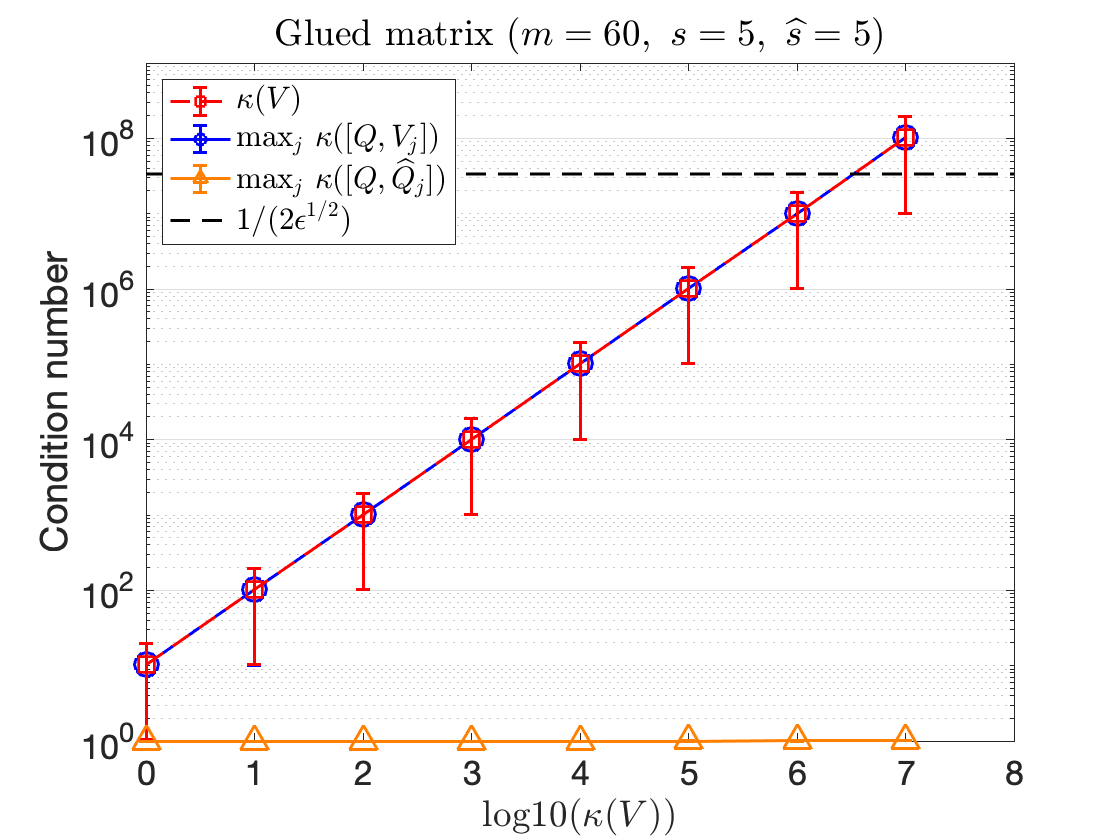}
    }
    \caption{Condition number with CholQR2.}\label{fig:cholqr2-cond}
  \end{subfigure}
  \begin{subfigure}[b]{.53\linewidth}
   \centerline{
     \includegraphics[width=\linewidth]{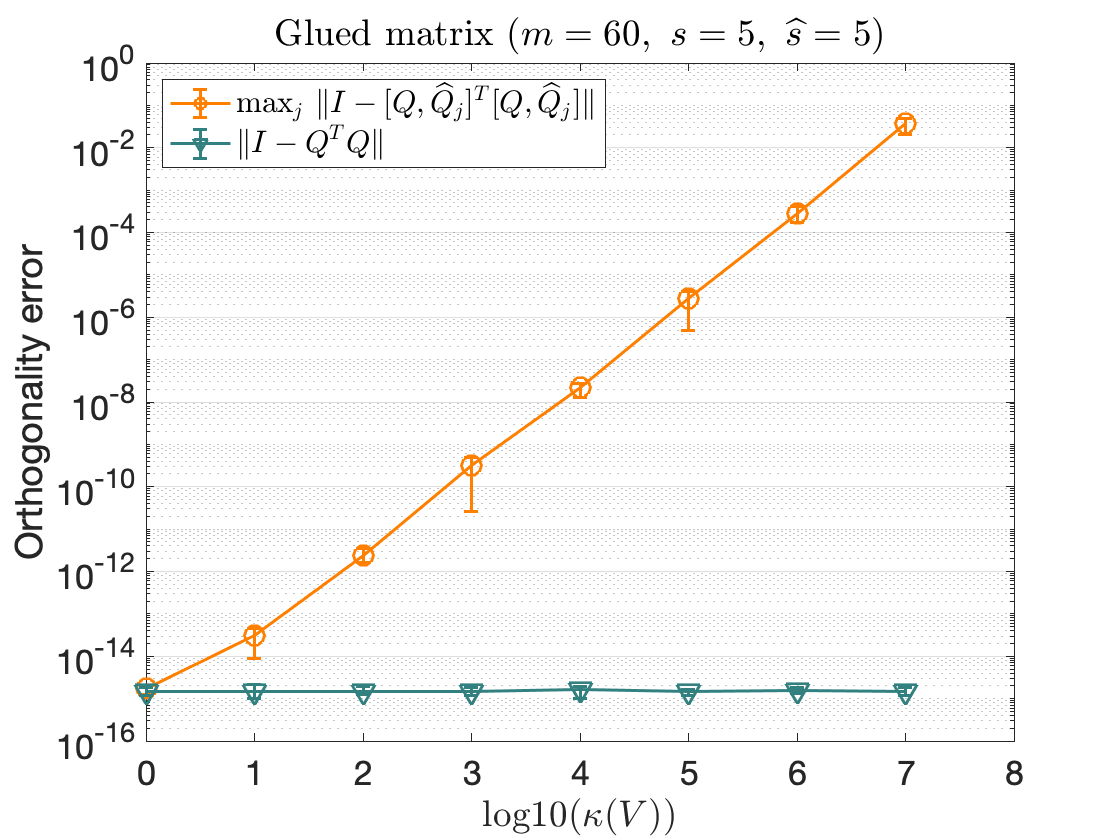}
    }
    \caption{Orthogonality errors with CholQR2.}\label{fig:cholqr2-orth}
  \end{subfigure}
}
\centerline{
  \begin{subfigure}[b]{.53\linewidth}
   \centerline{
     \includegraphics[width=\linewidth]{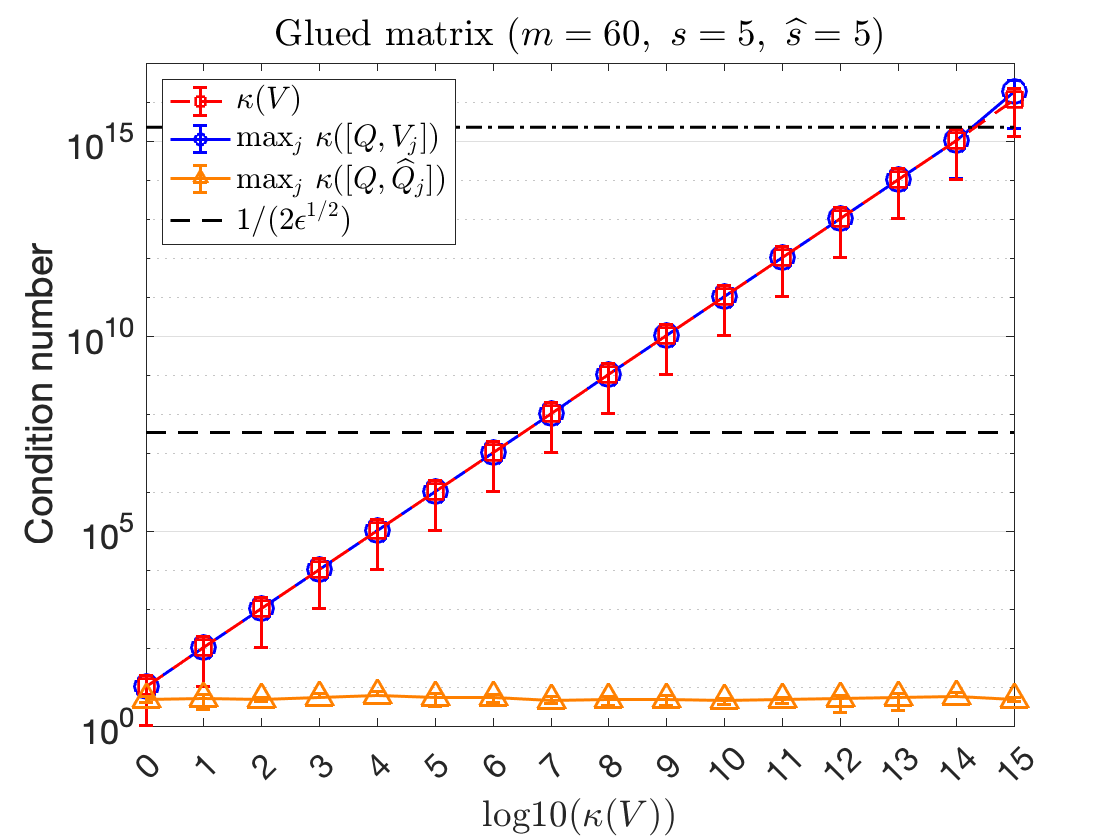}
    }
    \caption{Condition number with RandCholQR.}
  \end{subfigure}
  \begin{subfigure}[b]{.53\linewidth}
   \centerline{
     \includegraphics[width=\linewidth]{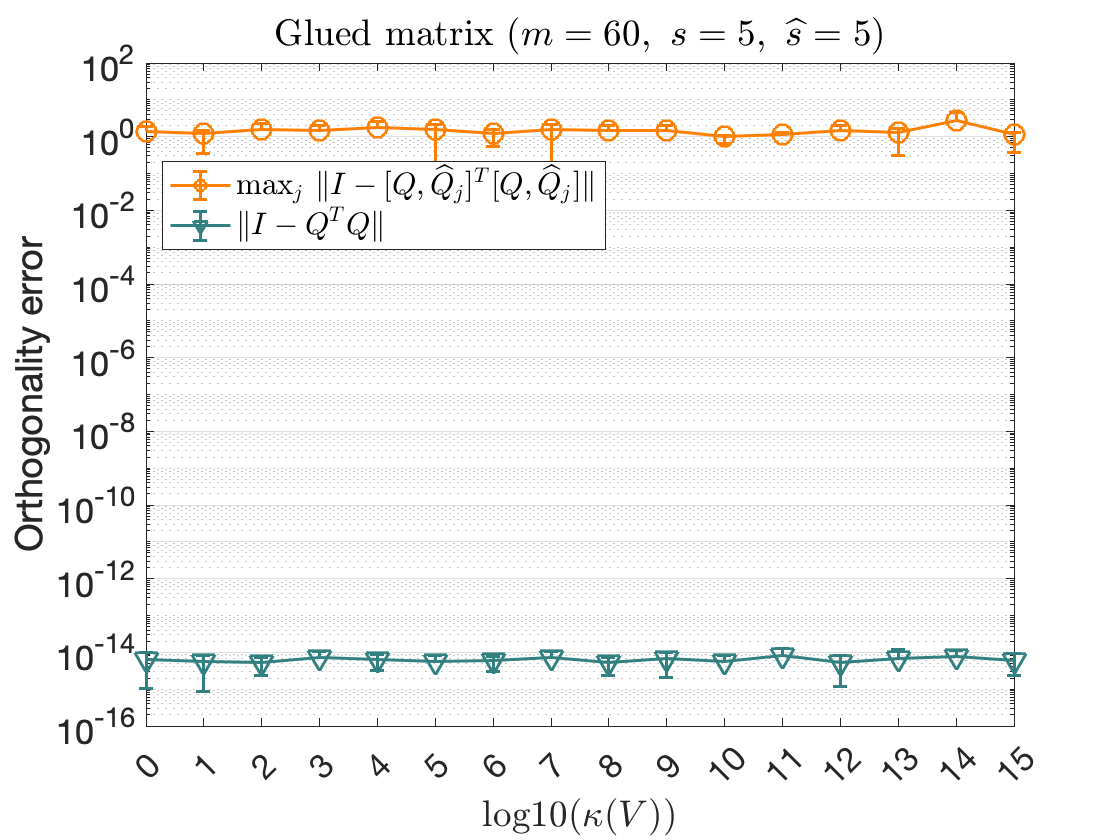}
    }
    \caption{Orthogonality errors with RandCholQR.}
  \end{subfigure}
}
 \caption{Condition number and orthogonality error with one-step BCGS2 with CholQR2 or RandCholQR first intra-block orthogonalization on {\tt glued} matrix, where Cholesky factorization of the Gram matrix $G$ in CholQR failed with non-positive pivots when $\kappa(V) \ge 10^8$ in Figures~\ref{fig:cholqr2-cond} and \ref{fig:cholqr2-orth}.}\label{fig:cholqr-error}
\end{figure}

We compared the orthogonality errors of the proposed block orthogonalization schemes using the default double precision in MATLAB. For these studies, instead of studying the numerical properties of the proposed methods within the $s$-step GMRES, we first treat them as general block orthogonalization schemes and use synthetic matrices as the input vectors. This allows us to control the condition number of the matrix easily. Numerical results showing how the condition numbers of the Krylov vectors could grow, can be found in~\cite{Yamazaki:2024}.

%
\ignore{
Finally, we study how the condition numbers grow for the basis vectors generated by MPK using various positive indefinite matrices of dimension between 200,000 and 300,000 from the SuiteSparse Matrix Collection\footnote{\url{https://sparse.tamu.edu}} (Figure\ref{fig:cond-mpk}). In order to maintain the stability of the original $s$-step method, we scaled the columns and then rows of the matrices by the maximum nonzero entries in the columns and rows (hence, all the resulting matrices are non-symmetric).
For all of our experiments with MPK and $s$-step GMRES, we used monomial basis vectors, even though using more stable bases, like Newton or Chebyshev bases, could reduce the condition number and improve the applicability of our approaches to a wider class of problems.
}

We first study how the orthogonality errors grow with the condition numbers of the input vectors for the one-stage algorithms. 
Figure~\ref{fig:cholqr-error} shows the orthogonality error when CholQR2 or RandCholQR
are used for the first intra block orthogonalization of the one-stage BCGS2.
Our test matrix is the {\tt glued} matrix~\cite{Smoktunowicz:2006} that has the same specified order of the condition number for each panel $V_j$ and for the overall matrix $V_{1:12}$. As expected, the orthogonality error of the basis vectors $\widehat{Q}$ was $\mathcal{O}(\epsilon)$ using CholQR2
when the condition number of the input matrix is smaller than $\mathcal{O}(\epsilon)^{-1/2}$.
With RandCholQR, the same $\mathcal{O}(\epsilon)$ orthogonality errors
were obtained as long as the input matrix is numerically full-rank with the condition number of $\mathcal{O}(\epsilon)^{-1}$, and hence demonstrating superior stability compared to CholQR2.

Next, we show that the two-stage approach obtains $\mathcal{O}(\epsilon)$ orthogonality error when the condition~\eqref{eq:assumption-3} or \eqref{eq:assumption-4} is satisfied, using BCGS-PIP or RandBCGS2, respectively.
Figure~\ref{fig:cond-2stage} shows the condition number of basis vectors using the two-stage approach with BCGS-PIP or RandBCGS2 as the preprocessing schemes, while Figure~\ref{fig:orth-2stage} shows the orthogonality errors using RandBCGS2. The test matrix is the {\tt glued} matrix, where each big panel has the condition number $\mathcal{O}(10^{15})$.
For this synthetic matrix, the Cholesky factorization of the Gram matrix failed with non-positive pivot in BCGS-PIP when the condition number of the accumulated panels increased more than $\mathcal{O}(\epsilon)^{-1/2}$. In contrast, RandBCGS2 managed to keep the $\mathcal{O}(1)$ condition number of the big panel $[Q_{1:\ell-1},\widehat{Q}_{\ell:t}]$, and the overall orthogonality error of $Q$ was $\mathcal{O}(\epsilon)$.

\begin{figure}[t]
  \begin{subfigure}[b]{\linewidth}
   \centerline{
     \includegraphics[width=.9\linewidth]{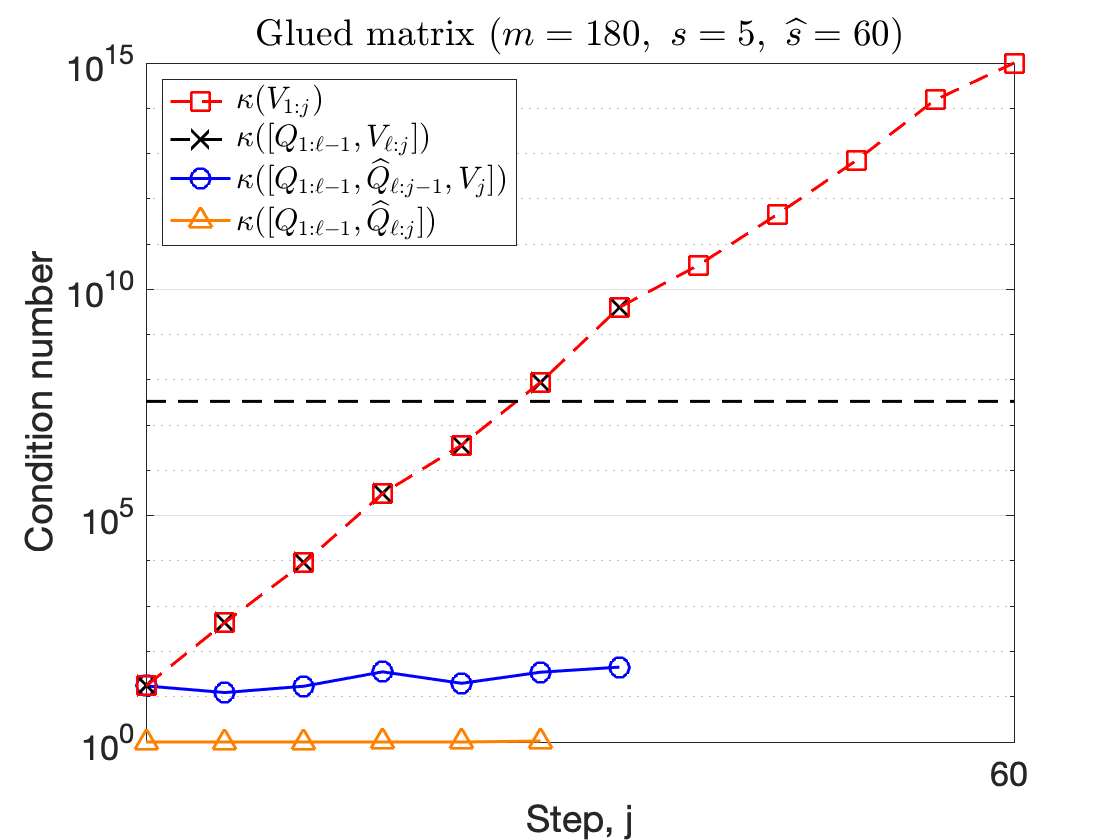}
    }
    \caption{BCGS-PIP.}
  \end{subfigure}

  \begin{subfigure}[b]{\linewidth}
   \centerline{
     \includegraphics[width=.9\linewidth]{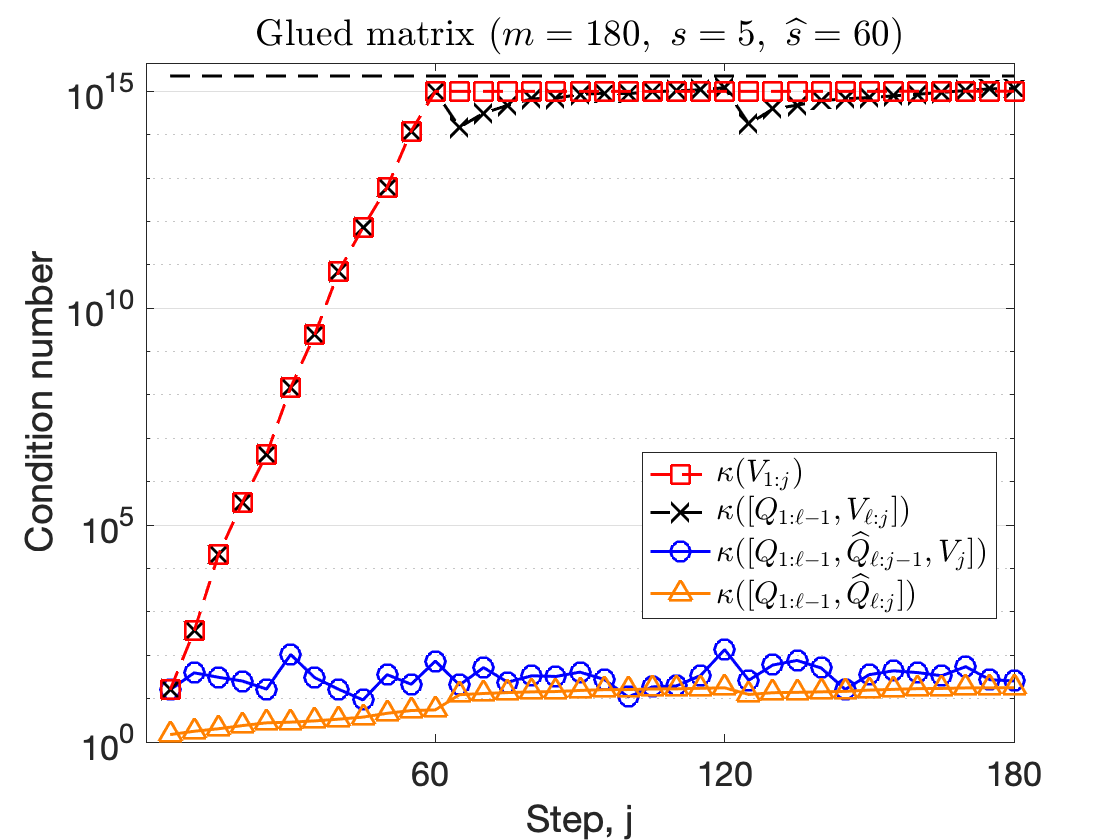}
    }
    \caption{RandBCGS2}\label{fig:cond-2stage}
  \end{subfigure}  
 \caption{Condition number (marker at every $s$ steps) using two-stage approach with BCGS-PIP and RandBCGS2 on {\tt glued} matrix with $(n,m,\hat{s},s)=(100000, 180, 60, 5)$.} \label{fig:cond-2stage}
\end{figure}

\begin{figure}[t]
   \centerline{
     \includegraphics[width=.9\linewidth]{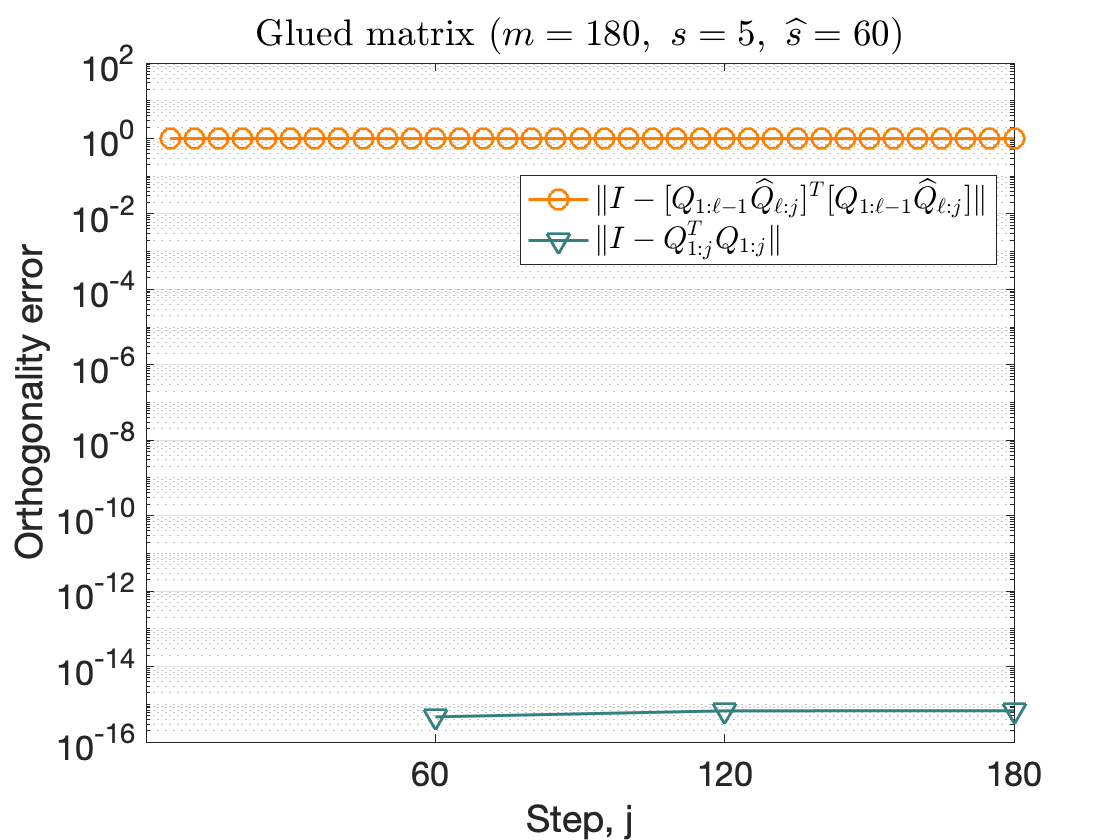}
    }
 \caption{Orthogonality error (orange circle marker at every $s$ steps, while green triangle marker at every $\widehat{s}$ steps) using two-stage approach with RandBCGS2 on {\tt glued} matrix with $(n,m,\hat{s},s)=(100000, 180, 60, 5)$.} \label{fig:orth-2stage}
\end{figure}

\begin{figure}[t]
\centerline{
     \includegraphics[width=.9\linewidth]{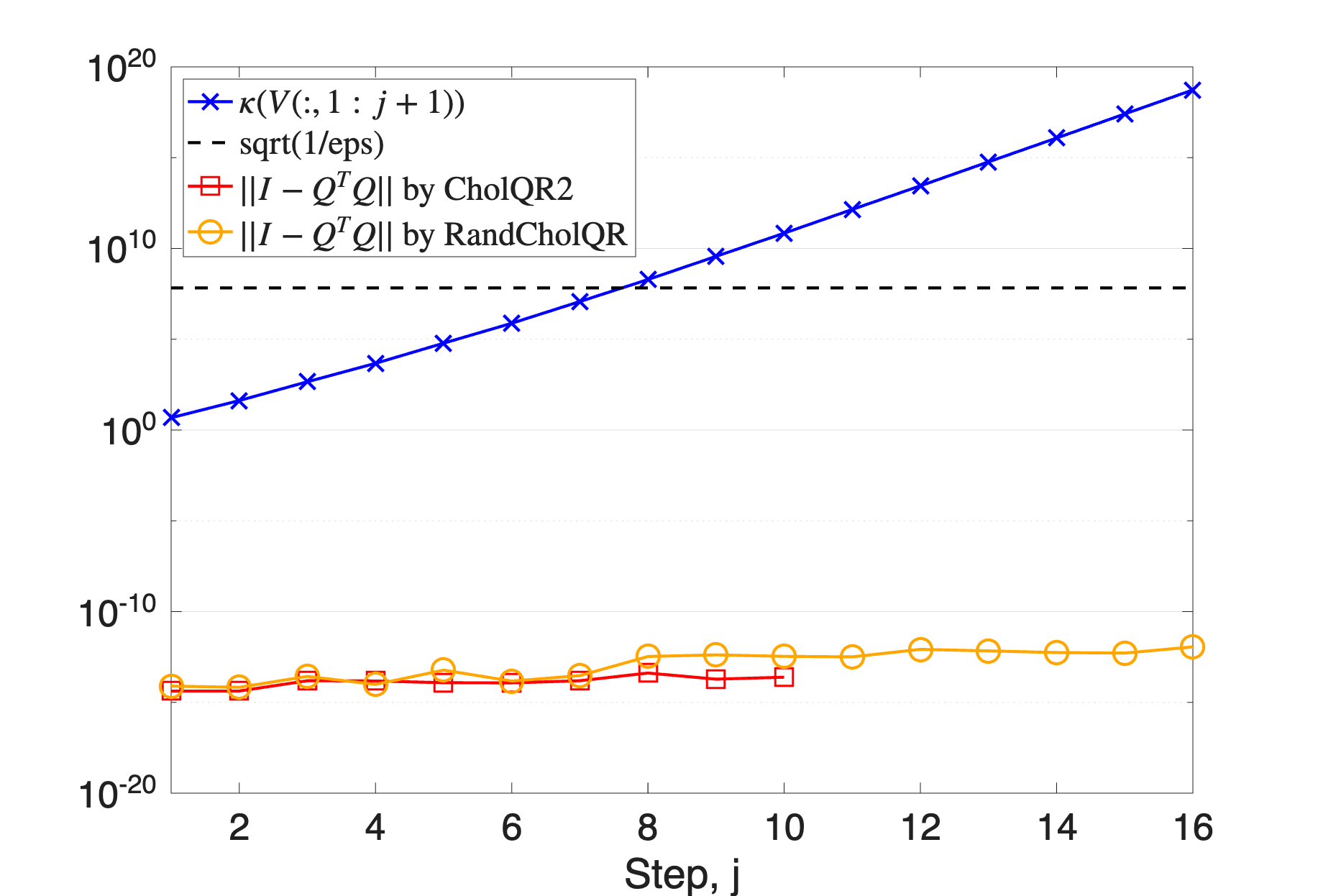}
}
 \caption{Condition number and orthogonality error with CholQR2 or RandCholQR where the matrix is generated by MPK for a 2D Laplace matrix of dimension $128^2$..}\label{fig:mpk-laplace}
\end{figure}

Finally, Figure~\ref{fig:mpk-laplace} shows the condition number and the orthogonality errors based on CholQR2 and RandCholQR where the basis vectors $V_{1:j+1}$ are generated by MPK. In particular, we generated $V_{1:j+1}$ such that $v_{j+1} := A v_j$ where $A$ is a 2D Laplacian matrix of dimension $128^2$, $v_1 = b / \|b\|$, $b := A x$, and the numerical values of all $x$'s entries are one. The condition number of $V_{1:j+1}$ grew quickly and CholQR2 failed when $j=11$, while RandCholQR could generate the orthonormal basis vectors as long as the input basis vectors are numerically full rank. Even though we used the step size of $s=5$ for all of our performance studies, these numerical results indicate that we could have used a larger step size with the random sketching techniques and potentially obtained higher performance of $s$-step GMRES.

\section{Implementation}
\label{sec:imple}
To study the performance of the block orthogonalization algorithms for $s$-step GMRES running on a GPU cluster,
we have implemented these algorithms within the Trilinos software framework \cite{Trilinos:2005,trilinos-website}.
Trilinos is a collection of open-source software libraries, called packages, 
for solving linear, non-linear, optimization, and uncertainty quantification problems.
It is intended to be used as building blocks for developing large-scale scientific or engineering applications.
Hence, any improvement in the solver performance could have direct impacts to the application performance.
In addition, Trilinos software stack provides portable performance of the solver on different hardware architectures,
with a single code base.
In particular, our implementation is based on Tpetra~\cite{tpetra,tpetra-website}
for distributed matrix and vector operations and Kokkos-Kernels~\cite{KK}
for the on-node portable matrix and vector operations
(which also provides the interfaces for the vendor-optimized
 kernels like NVIDIA cuBLAS, cuSparse, and cuSolver).

On a GPU cluster, our GMRES implementation uses GPUs to generate the orthonormal basis vectors.
The coefficient matrix $A$ and Krylov vectors $V$ are distributed among MPI processes in 1D block row format
(e.g., using a graph partitoner like ParMETIS), where $A$ is locally stored as the Compressed Sparse Row (CSR) format, while $V$ is stored in the column-major format. The operations with the small projected matrices, including solving a small least-squares problem, is redundantly done on CPU by each MPI process.

Our focus is on the block orthogonalization of the vectors,
which are distributed in 1D block row format among the MPI processes.
The orthogonalization process mainly consists of
dot-products, vector updates, and vector scaling 
(e.g., $R_{1:j-1,j} := Q_{1:j-1}^TV_j$ and $\widehat{V}_j := V_j - Q_{1:j-1}R_{1:j-1,j}$ of BCGS in Figure~\ref{algo:bcgs}, 
   and $\widehat{Q}_j := \widehat{V}_j R_{j,j}^{-1}$ of CholQR in Figure~\ref{algo:cholqr}, respectively). 
The dot-products $Q_{1:j-1}^TV_j$ requires a global reduce among all the MPI processes,
and the resulting matrix $R_{1:j-1, j}$ is stored redundantly on the CPU by all the MPI processes.
Given the upper-triangular matrix on each MPI process, the vectors can be updated and scaled locally without any additional 
communication.  All the local computations are performed by the computational kernels through Kokkos Kernels, either on a CPU or on a GPU.

We have implemented the random-sketching using standard linear algebra kernels
as discussed in Section~\ref{sec:randCholQR}. 
These distributed-memory or on-node kernels are readily available through Tpetra or Kokkos-Kernels, given that the random-sketching matrix $\boldsymbol{\Theta}$ is explicitly generated and stored in memory.
\begin{itemize}
\item For Gaussian sketch, the dense sketching matrix $\Theta$ is distributed among the MPI processes in the 1D block row format, and each MPI process stores the local matrix in the column major order.
\item For Count sketch, the sparse sketching matrix $\Theta$ is also distributed in the 1D block row format, where each local sparse matrix is stored in the CSR format.
\item For Count-Gaussian sketching, the sparse Count-sketching matrix $\Theta_c$ is distributed in the 1D block row format, but the dense Gaussian-sketching matrix $\Theta_g$ is duplicated on all the MPI processes such that it can be applied before performing the global-reduce of the final sketched basis vectors.
\end{itemize}

\section{Performance Experiments}
\label{sec:perf}

We conducted our performance tests on the Perlmutter supercomputer at National Energy Research Scientific Computing (NERSC) Center. Each compute node of Perlmutter
has one 64-core AMD EPYC 7763 CPUs and four NVIDIA A100 GPUs.
On each node, we launched 4 MPI processes (one MPI per GPU) and assigned 16 CPU cores to each MPI process.
For all the performance studies with one-stage or two-stage algorithm,
we used the sketch size of $2s$ or $2\widehat{s}$ for the Gaussian Sketch, 
while the sketch size of $2s^2$ or  $2\widehat{s}^2$ is used for the Count Sketch,
respectively.

The code was compiled using Cray's compiler wrapper 
with Cray LibSci version 23.2, CUDA version 11.5
and Cray MPICH version 8.1.
The GPU-aware MPI was not available on Perlmutter, and hence, all the MPI communications are performed through the CPU.
We configured Trilinos such that all the local dense and sparse matrix operations are performed using CUBLAS and CuSparse, respectively.

\subsection{Single-GPU Sketching Performance (Gaussian vs. Count)}

\begin{figure}[t]
   \begin{subfigure}[b]{\linewidth}
   \centerline{
     \includegraphics[width=.85\linewidth]{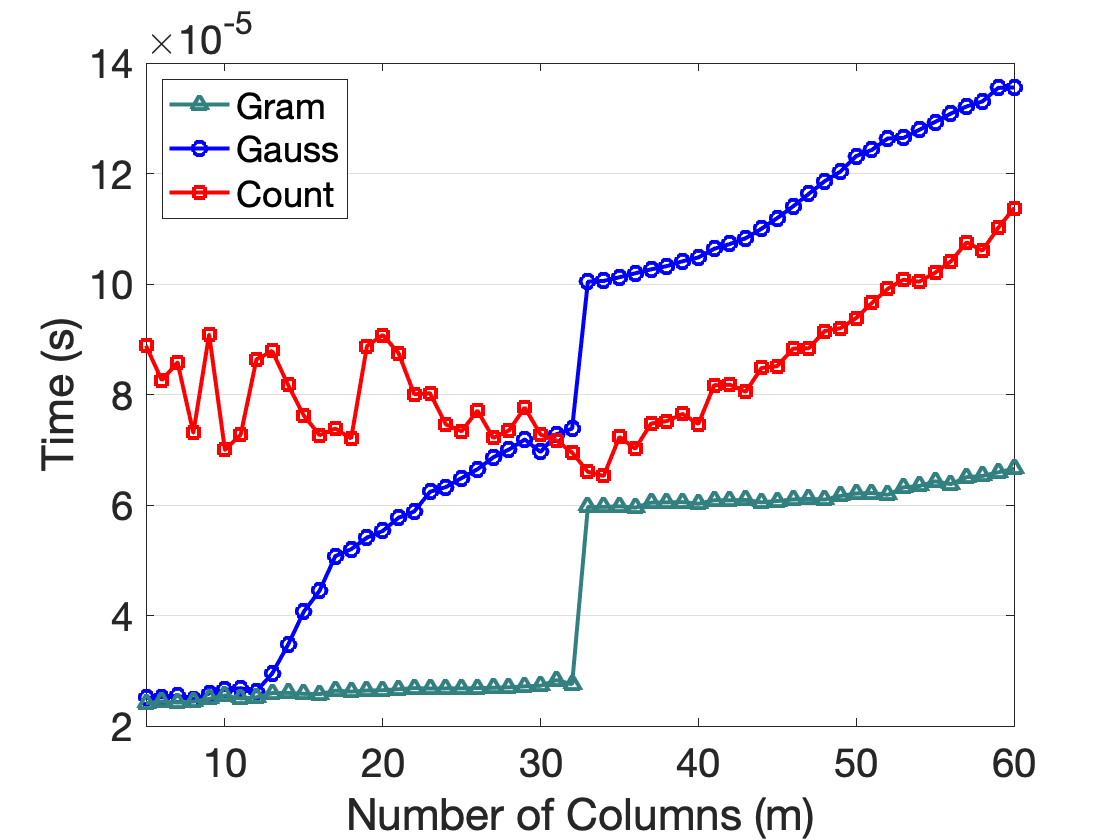}
   }
   \caption{Time in seconds.} \label{fig:sketch-time}
   \end{subfigure}   
   \begin{subfigure}[b]{\linewidth}
   \centerline{
     \includegraphics[width=.85\linewidth]{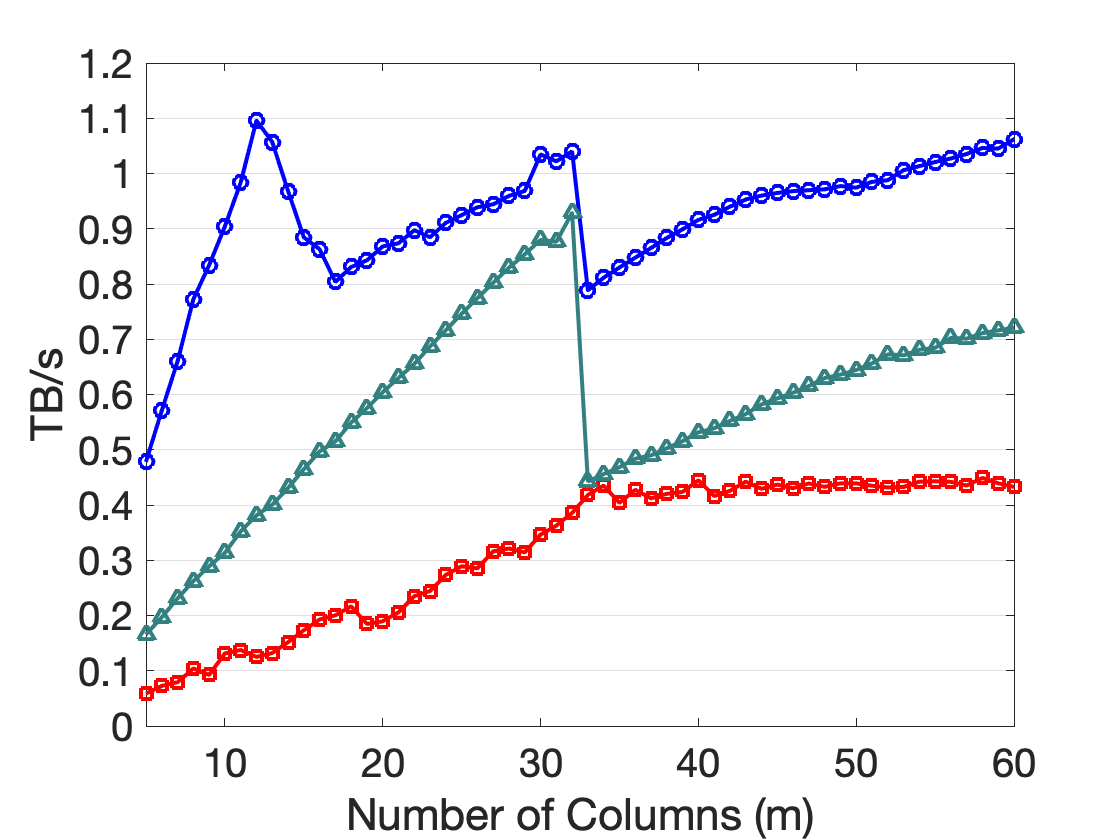}
   }
   \caption{Bandwidth.} \label{fig:sketch-bw}
   \end{subfigure}   
\caption{Sketch performance on one NVIDIA A100 GPU.}
\end{figure}

\begin{figure}[h]
   \begin{subfigure}[b]{\linewidth}
   \centerline{
     \includegraphics[width=.85\linewidth]{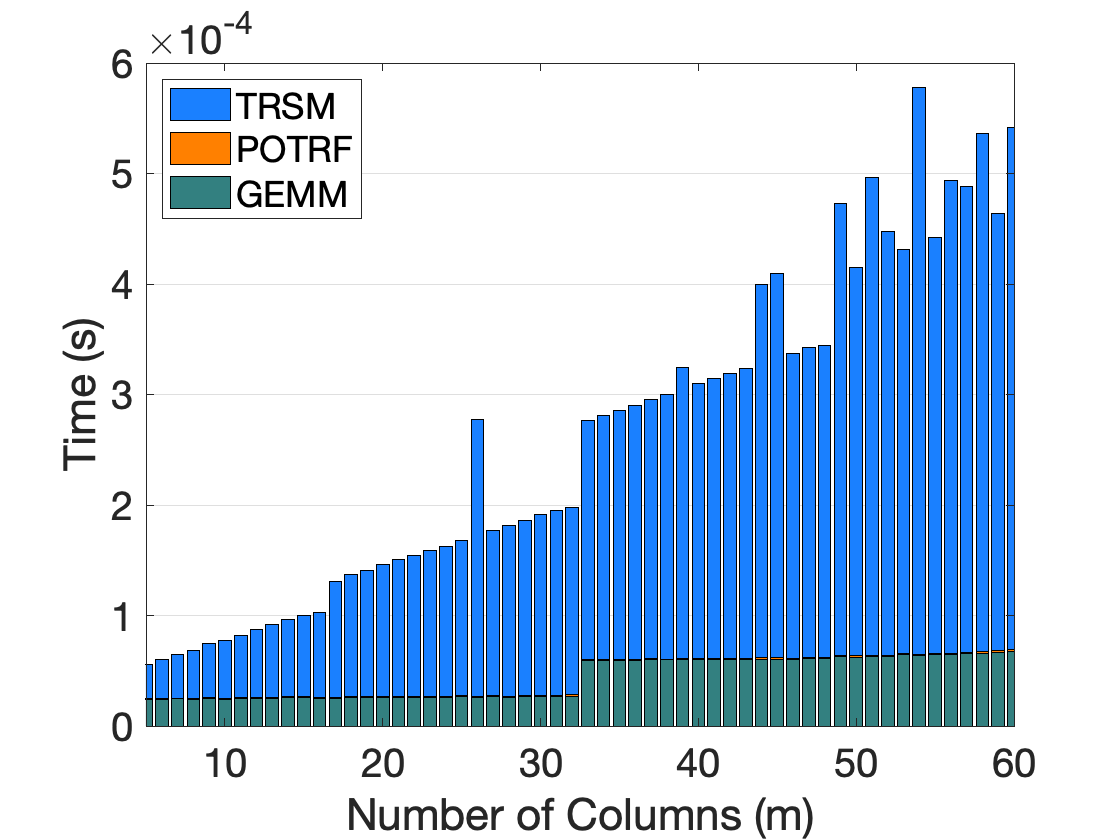}
   }
   \caption{Time in seconds.}
  \end{subfigure}
   \begin{subfigure}[b]{\linewidth}
   \centerline{
     \includegraphics[width=.85\linewidth]{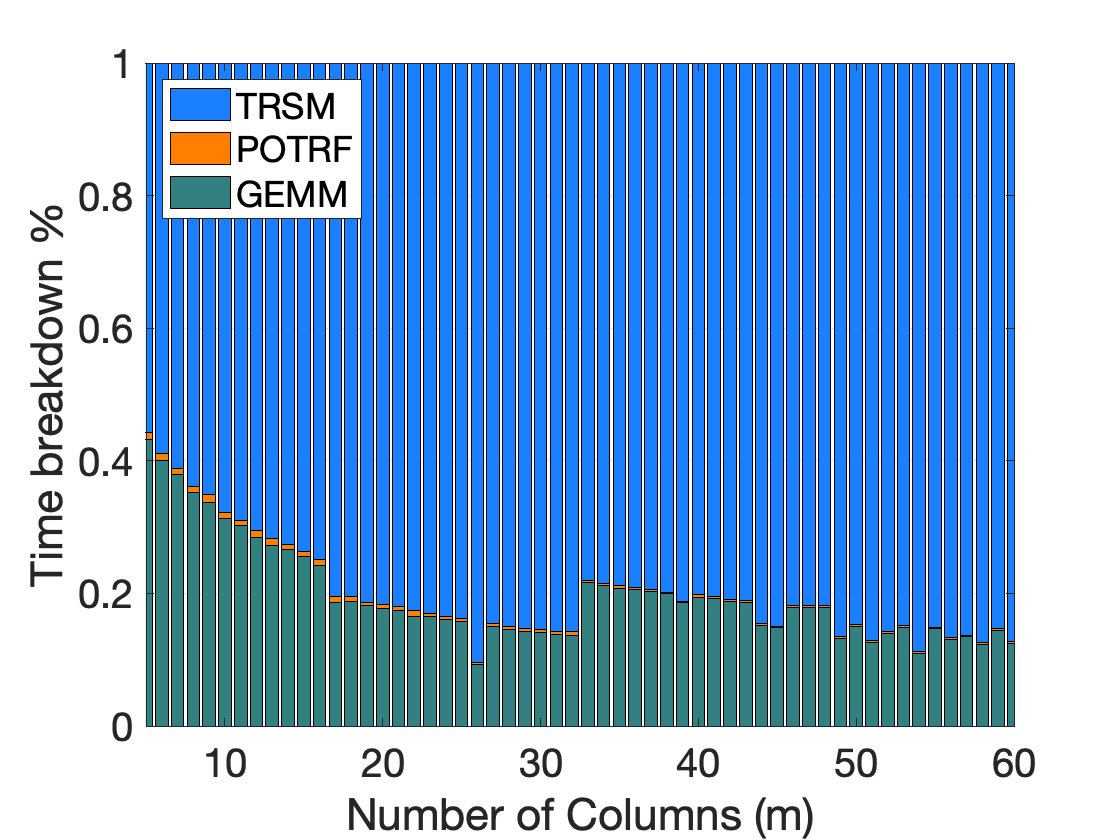}
   }
   \caption{Ratio.}
  \end{subfigure}
\caption{Time breakdown for CholQR on one NVIDIA A100 GPU.} \label{fig:cholqr-time}
\end{figure}

Figure~\ref{fig:sketch-time} compares the time required for Gaussian and Count sketch on a single NVIDIA A100 GPU, with the time required for computing the Gram matrix of the block vectors (e.g., needed for CholQR), 
where the number of rows is fixed at $n=10^5$, while the number of columns, $m$, is increased from 5 to 60. The Count Sketch was slower than the Gaussian Sketch for a small number of columns (e.g., $s=5$ for the one-stage block orthogonalization). However, because the Count Sketch has the computational complexity of $\mathcal{O}(nm)$ compared to $\mathcal{O}(nm^2)$ of Gaussian Sketch, the time required for the Count Sketch increased less with the increasing number of columns, and it became faster than the Gaussian Sketch for a large enough number of columns (e.g., $\widehat{s}=60$ for the two-stage orthogonalization, though we sketch $s$ columns at a time).

Figure~\ref{fig:sketch-bw} shows the observed bandwidth for the three algorithms with $4\mbox{ bytes}$ to store the column index for the sparse CSR format and $8\mbox{ bytes}$ to store the double-precision numerical value for both the dense vectors and sparse matrix, respectively. We let the amount of the data required to move to be $8\mbox{ bytes}\cdot mn$ for computing the Gram matrix (reading $V$ once), and $8\mbox{ bytes}\cdot(3mn)$ or $8\mbox{ bytes}\cdot mn + (8+4)\mbox{ bytes}\cdot m$ for the Gaussian Sketch or Count Sketch (reading $V$ and dense or sparse $\Theta$ once), respectively. As we expected, due to the irregular data access, the Count Sketch obtained the lower bandwidth than the Gaussian Sketch (which obtained close to the NVIDIA A100 memory bandwidth, 1.5TBytes per second). The bandwidth obtained for computing the Gram matrix was about the half of that observed for the Gaussian Sketch. This could be because Tpetra uses non-symmetric dense matrix-matrix multiply (GEMM) for computing the Gram matrix, potentially doubling the amount of the required data traffic.

\subsection{Breakdown of Orthogonalization Time on one and multiple GPUs (Gaussian, Count, vs. CountGauss)}

We now study the performance of the block orthogonalization kernels. Figure~\ref{fig:cholqr-time} shows the breakdown of time needed for CholQR on a single NVIDIA A100 GPU. With our experiment setups, the dense triangular solves (TRSM) required for generating the orthogonal basis vectors took longer than the dense matrix-matrix vector multiply (GEMM) needed to generate the Gram matrix, even though their computational complexity costs are about the same. The time needed to compute the Cholesky factorization of the Gram matrix on the CPU was negligible.

\begin{figure}[t]
   \begin{subfigure}[b]{\linewidth}
   \centerline{
     \includegraphics[width=.85\linewidth]{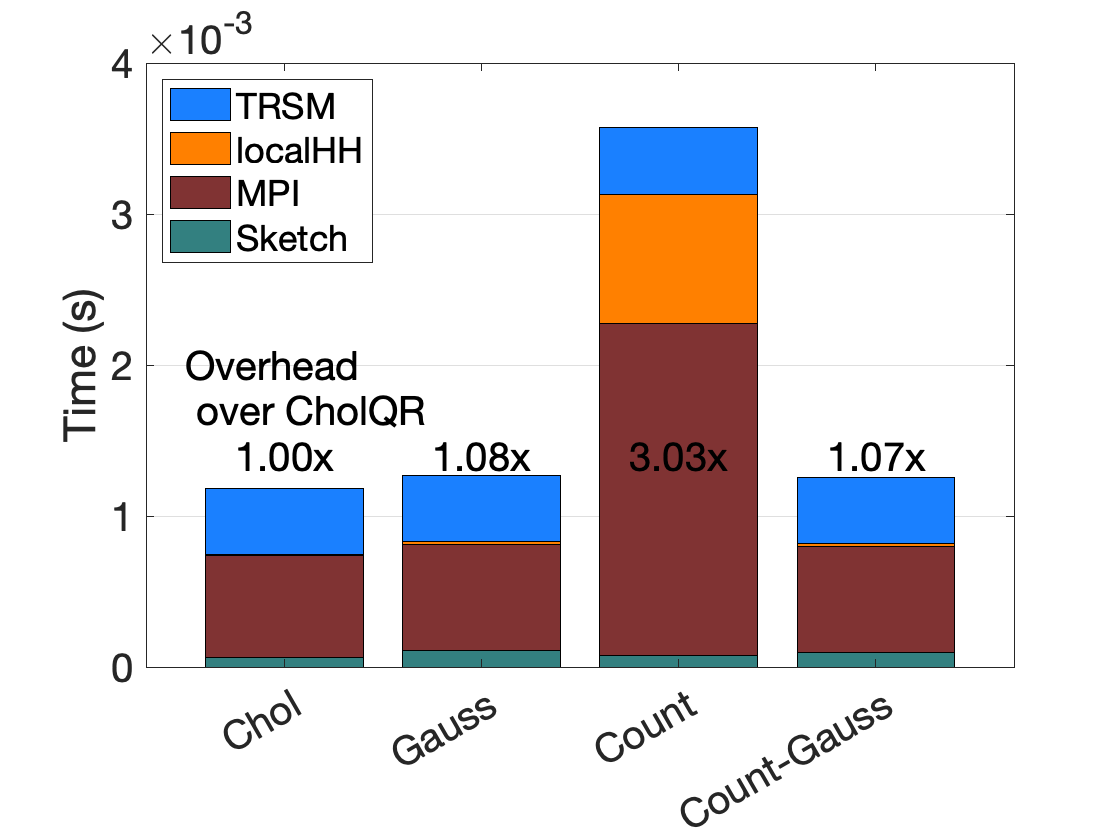}
   }
   \caption{Breakdown of Intra-block orthogonalization time ($m=40$).} \label{fig:distributed-intra-time}
  \end{subfigure}   
   \begin{subfigure}[b]{\linewidth}
   \centerline{
     \includegraphics[width=.85\linewidth]{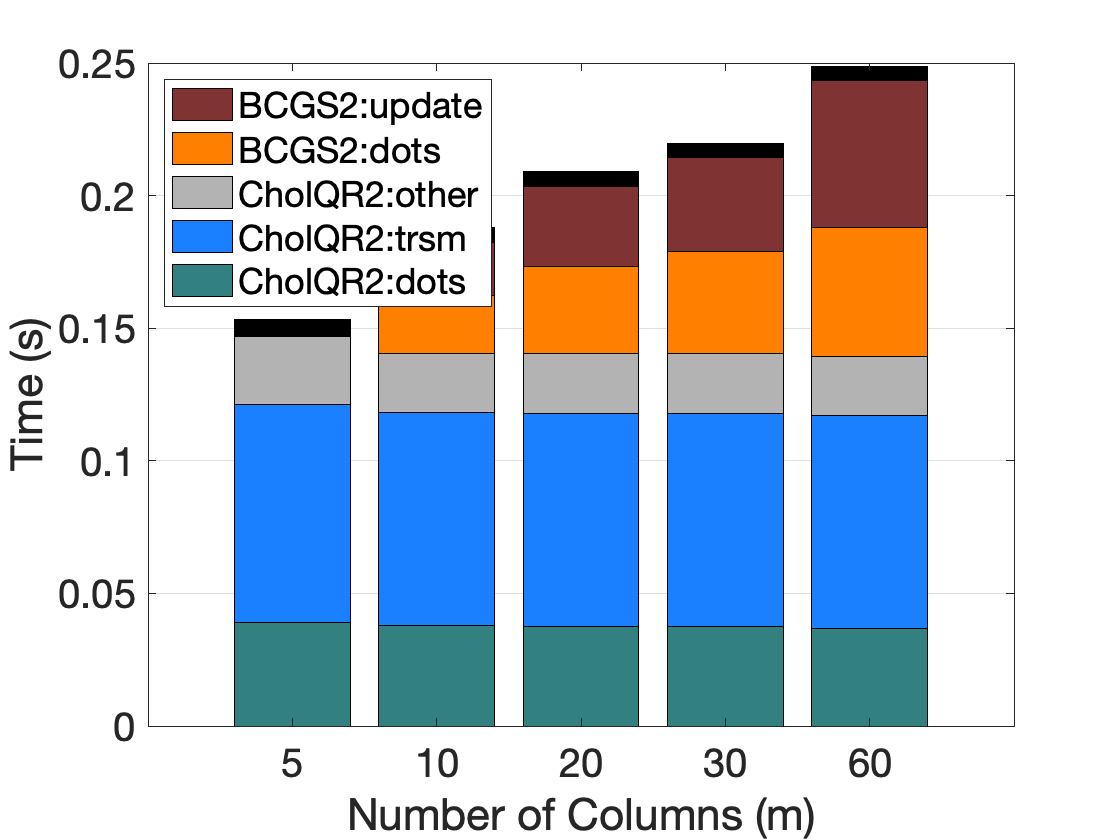}
   }
   \caption{BCGS2 with CholQR2.} \label{fig:distributed-bcgs-time}
  \end{subfigure}
\caption{Breakdown of orthogonalization time on four NVIDIA A100 GPUs. }
\end{figure}

\begin{table*}
\scriptsize
 \centerline{
 \begin{tabular}{c|rrr|rrr|rrr|rrr|rrr}
  & \multicolumn{3}{c}{GMRES + ICGS (1659)} & \multicolumn{3}{|c}{$s$-step + CholQR2 (1660)} & \multicolumn{3}{|c}{$s$-step + RandQR (1660)} & \multicolumn{3}{|c}{Two-stage + PIP (1700)}  & \multicolumn{3}{|c}{Two-stage + RandBCGS2 (1700)}\\\
  \# nodes & SpMV & Ortho & Total & SpMV & Ortho     & Total       & SpMV & Ortho     & Total        & SpMV & Ortho & Total & SpMV & Ortho & Total \\
  \hline\hline
  1        & 7.96 & 35.00 & 40.81 & 7.38 & 9.14+8.42 & 23.95       & 7.42 & 9.11+8.55 & 24.08        & 8.14 & 9.87  & 16.92& 9.10 & 14.69 & 21.58\\ 
           &      &       &       &   & 1.99$\times$ & 1.70$\times$&   & 1.98$\times$ & 1.69$\times$ & & 3.55$\times$ & 2.41$\times$ & & 2.38$\times$ & 1.89$\times$\\
  \hline
  2        & 6.26 & 22.34 & 26.47 & 6.28 & 5.97+4.96 & 15.65   & 6.30 & 6.05+5.04     & 15.75        & 6.48 & 5.41  & 10.91 & 7.09 & 7.85  & 13.28\\
           &      &       &       &   & 2.04$\times$ & 1.69$\times$&  & 2.01$\times$  & 1.68$\times$ &&  4.13$\times$ & 2.43$\times$ && 2.85$\times$ & 1.99$\times$\\
  \hline
  4        & 5.13 & 16.56 & 19.85 & 5.08 & 4.10+3.54 & 11.45   & 5.06 & 4.14+3.66     & 11.54        & 5.24 & 3.59  & 7.87  & 5.74 & 4.92  & 9.25\\
           &      &       &       &   & 2.17$\times$ & 1.73$\times$& & 2.12$\times$   & 1.72$\times$ && 4.61$\times$ & 2.52$\times$ && 3.37$\times$ & 2.15$\times$\\
  \hline
  8        & 4.47 & 14.46 & 17.31 & 4.47 & 3.17+2.88 & 9.43    & 4.43 & 3.13+2.99     & 9.44         & 4.44 & 2.54  & 6.21  & 4.66 & 3.09  & 6.94\\
           &      &       &       &   & 2.39$\times$ & 1.84$\times$& & 2.36$\times$   & 1.93$\times$ && 5.69$\times$ & 2.79$\times$ && 4.68$\times$ & 2.49$\times$\\
  \hline
  16       & 4.41 & 13.40 & 15.78 & 4.06 & 2.69+2.43 & 8.24    & 4.06 & 2.70+2.51     & 8.32         & 4.15 & 2.24  & 5.50  & 4.26 & 2.66  & 5.89\\
           &      &       &       &   & 2.62$\times$ & 1.92$\times$& & 2.57$\times$   & 1.90$\times$ && 5.98$\times$ & 2.87$\times$ && 5.04$\times$ & 2.68$\times$\\
 \end{tabular}
 }
  \caption{
           Parallel Strong Scaling of time-to-solution with 7-points 3D Laplace, $n = 300^3$. On each node, we launched 4 MPI processes (one MPI per GPU).
           The table also shows the speedup gained
           using $s$-step ($s=5$) and two-stage ($\widehat{s}=m$) over standard GMRES ($m=100$) for orthogonalization and total solution time.} \label{tab:time-to-sol}
 \end{table*}

Figure~\ref{fig:distributed-intra-time} shows the breakdown of the intra-block orthogonalization time on the four NVIDIA A100 GPUs available on a single Perlmutter compute node.
Compared to the Gaussian Sketch, the ``sketch time" (i.e., the time to apply the sketch) of Count-Sketch was slightly faster, but due to its larger sketch size, it required more time for the global-reduce and local Householder QR. Overall, Count-Gaussian sketch obtained the best performance, but the performance on Perlmutter was largely dominated by the global-reduce and TRSM (and not by the sketching time), and Gaussian and Count-Gaussian sketch obtained similar performance.

Figure~\ref{fig:distributed-bcgs-time} then shows the breakdown of the BCGS, with CholQR2, orthogonalization time on the four NVIDIA A100 GPUs. It shows the average time required by the $s$-step GMRES over 600 iterations to orthogonalize the basis vectors with $s=5$ for solving the 2D Laplace problem of dimension $700^2$, and hence the number of columns refers to the GMRES' restart cycle. The inter-block orthogonalization became more significant as the restart-cycle length was increased. However, the TRSM time was still the most dominant part of the overall orthogonalization time.

\subsection{$s$-step GMRES Strong-scaling Results}

Although the breakdown of the iteration time depends on the matrix,
in Table~\ref{tab:time-to-sol}, we show the parallel strong-scaling performance of the $s$-step GMRES for solving a 3D Laplace problem,
from which we can infer the performance for other problems (please see, for instance \cite{Yamazaki:2024}, for the performance using difference matrices).
We used the restart length of 100 (i.e., $m=100$), and
considered GMRES to have converged when the relative residual
norm is reduced by six orders of magnitude.

For the one-stage algorithm, compared to CholQR, the random sketching had virtually no overhead, while improving the stability of the orthogonalization as shown in Section~\ref{sec:numeric}.
Compared to the one-stage algorithm, the random sketching had more significant overhead for the two-stage orthogonalization due to the larger sketch size. However, the overhead became less significant as we increased the number of MPI processes, and the latency cost became more significant, with 1.49, 1.45, 1.37, 1.22, and 1.19$\times$ overhead on 1, 2, 4, 8, and 16 nodes, respectively.

 

\section{Conclusion}

We integrated random sketching techniques into the block orthogonalization process, required for the $s$-step GMRES.
The resulting algorithm ensures that the overall orthogonalization errors are bounded by the machine precision as long as each of the block vectors are numerically full-rank.
Our performance results demonstrated that
the numerical stability of the block orthogonalization process is improved with a relatively small performance overhead.
Our implementation of the random sketching utilizes standard linear algebra kernels such that it is portable to different computer architectures.
Though the vendor-optimized versions of these kernels are often available, they may not be optimized for the specific shapes or sparsity patterns of the sketching matrices, and we have observed the sparse sketch often obtains the suboptimal performance.
Nevertheless,
the sparse random sketching has the complexity of $\mathcal{O}(n)$, and with a careful implementation, RandCholQR, using a spares sketch, might not only enhance the numerical stability but also be able to outperform CholQR2.


\section*{Appendix}
\label{sec:appendix}
Here we provide a sketch of the proof of Proposition~\ref{prop72}.

\begin{proof}
Because of Proposition \ref{prop_2stage}, 
it is sufficient to prove that the condition number of the big panel $\overline{Q}_{\ell+1:t+1}$,
before the CholQR (on Line 7 of Figure \ref{algo:two-stage-flat}), is bounded as
\begin{equation}\label{ap:Qhatcond}
  \kappa\left( \overline{Q}_{\ell+1:t+1}\right) = O(1).
\end{equation}
Then with the bound~\eqref{eq:cholqr}, after the CholQR, the resulting big panel has the orthogonality error,
$$
\|I - Q_{\ell+1:t+1}^TQ_{\ell+1:t+1}\| = O(\epsilon),
$$
and by an argument identical to~\cite[Theorem 6.1]{Barlow:2024}, the overall stability of the two-stage BCGS2 is ensured with the $\mathcal{O}(\epsilon)$ orthogonality error among all the generated basis vectors,
$$
\|I - Q_{1:t+1}^TQ_{1:t+1}\| = O(\epsilon),
$$
up to a constant term $c(\epsilon,n,\widehat{s})$.

The general strategy for proving \eqref{ap:Qhatcond} is based on the observation that RandBCGS2 is identical to RandHH except that BCGS2-HH (Line 5 of Figure \ref{algo:incRand2}) is used to orthogonalize the sketched vectors $\widehat{\mathbf{v}}_{1:t+1}$ , instead of HH (Line 4 of Figure \ref{algo:randQR_nochol}).
For instance, the Lines 8--11 of Figure \ref{algo:incRand2} are equivalent to applying the forward-substitution to the big panel $\overline{Q}_{1:t+1} = \widehat{V}_{1:t+1}R_{1:t+1,1:t+1}^{-1}$. 
Hence, if we show that the backward error incurred during BCGS2-HH differs from HH by a constant factor $c(\epsilon,n,\widehat{s})$, then we can use the error analysis of RandHH from \cite{Higgins:2025} to prove \eqref{ap:Qhatcond}.

More specifically, in \cite[Corollary 5.1]{Higgins:2025}, it was proven that RandHH results in $\kappa\left( \overline{Q}_{\ell+1:t+1}\right) = O(1)$. This proof relies on \cite[Sections 5.2.1--5.2.8]{Higgins:2025}. Now, the error analysis in the first three subsections \cite[Sections 5.2.1--5.2.3]{Higgins:2025} applies to RandBCGS2 because they do not depend on the results of the QR factorization. Furthermore, the analysis of the backward error with the QR factorization in \cite[Section 5.2.4]{Higgins:2025} only differs by a constant factor if we prove the backward error of the QR factorization via BCGS2-HH only differs from the Householder QR backward error by a constant factor. Hence, the remaining error analysis in \cite[Sections 5.2.5--5.2.8]{Higgins:2025} applies to RandBCGS2, which uses BCGS2-HH in place of Householder QR, up to a constant factor. Therefore, by \cite[Corollary 5.1]{Higgins:2025}, we have $\kappa\left( \overline{Q}_{\ell+1:t+1}\right) = O(1)$ with RandBCGS2, thereby completing the proof. To simplify the notations, for the rest of the proof, we drop the accents on top of the computed sketched vectors (e.g., $\mathbf{v}$ instead of $\widehat{\mathbf{v}}$)

The main component of the proof is to derive the bound
$$
  \|\mathbf{v}_{1:j} - \mathbf{\bar{q}}_{1:j}R_{1:j,1:j}\|_2 = O(\epsilon) \|\mathbf{v}_{1:j}\|_2,
$$ 
where $\mathbf{\bar{q}}_{1:j}$ is an \emph{exactly} orthogonal matrix and $R_{1:j,1:j}$ is the upper-triangular matrix computed by BCGS2-HH, since such bound was proved for HH in \cite[Section 5.2.4]{Higgins:2025}. Since such result is not directly available in \cite{Barlow:2024}, we show how to obtain such a bound below.

According to \cite[Theorem 6.1]{Barlow:2024}, 
we first note that our computed matrices $\mathbf{q}_{1:j}$ and $R_{1:j,1:j}$ satisfy
\begin{align}
    \mathbf{v}_{1:j} + \Delta \mathbf{v}_{1:j} &= \mathbf{q}_{1:j}R_{1:j,1:j}, \label{eq:barlow1} \\
    \|\Delta \mathbf{v}_{1:j}\|_2 & = c_5(\epsilon,n,\widehat{s}) \|\mathbf{v}_{1:j}\|_2, \label{eq:barlow3} \\
    \|I - \mathbf{q}_{1:j}^T\mathbf{q}_{1:j}\|_2 &= c_6(\epsilon,n,\widehat{s}),\label{eq:barlow2}
\end{align}
for some constants $c_5(\epsilon,n,\widehat{s})$ and $ c_6(\epsilon,n,\widehat{s})$. 

Given the singular value decomposition of $\mathbf{q}_{1:j} = U\Sigma V^T$,
its perturbation $\Delta \mathbf{q}_{1:j}$ from the exactly-orthonormal vectors $\bar{\mathbf{q}}_{1:j}$
(i.e., $\Delta \mathbf{q}_{1:j} = \bar{\mathbf{q}}_{1:j} - \mathbf{q}_{1:j}$) can be bounded as
\begin{eqnarray}
\nonumber
    \|\Delta \mathbf{q}_{1:j}\|_2 
        & = & \|\bar{\mathbf{q}}_{1:j} - \mathbf{q}_{1:j}\|_2\\
\nonumber
        & = & \|U(I - \Sigma)V^T\|_2 \\
        & = & \|I - \Sigma\|_2 \leq c_7(\epsilon,n,\widehat{s}).\label{eq:c7}
\end{eqnarray} 
where the singular values of the exactly-orthonormal vectors $\bar{\mathbf{q}}_{1:j}$ are all ones (i.e., $\bar{\mathbf{q}}_{1:j} = U V^T$).
Using Weyl's inequality, \eqref{eq:barlow2} implies that for each $i \in \{1, \dots, j\}$, 
\begin{equation}
    \sqrt{1 - c_6(\epsilon,n,\widehat{s})} \leq \Sigma_{i,i} \leq \sqrt{1 + c_6(\epsilon,n,\widehat{s})}. \label{eq:svqBounds}
\end{equation}
Hence, the constant $c_7(\epsilon,n,\widehat{s})$ is given by
$$
  c_7(\epsilon,n,\widehat{s}) = \max\{ \sqrt{1 + c_6(\epsilon,n,\widehat{s})} - 1, 1 - \sqrt{1 - c_6(\epsilon,n,\widehat{s})} \} = O(\epsilon).
$$

Now, by substituting $\mathbf{q}_{1:j} = \bar{\mathbf{q}}_{1:j} - \Delta \mathbf{q}_{1:j}$ into \eqref{eq:barlow1},
we obtain
$$
  \bar{\mathbf{q}}_{1:j}R_{1:j,1:j} = \mathbf{v}_{1:j} + \Delta \mathbf{v}_{1:j} + \Delta \mathbf{q}_{1:j} R_{1:j,1:j},
$$ 
which allows us to derive the bound,
\begin{equation*}
    \|R_{1:j,1:j}\|_2 \leq \frac{1 + c_5(\epsilon,n,\widehat{s})}{1 - c_7(\epsilon,n,\widehat{s})} \| \mathbf{v}_{1:j}\|_2,
\end{equation*}
because
\begin{eqnarray}
\nonumber
    &&\hspace{-1cm}(1 - c_7(\epsilon,n,\widehat{s})) \|R_{1:j,1:j}\|_2\\
\nonumber
    &&\hspace{-1cm}\; = (1 - c_7(\epsilon,n,\widehat{s}))\|\bar{\mathbf{q}}_{1:j}R_{1:j,1:j}\|_2 \\
\nonumber
    &&\hspace{-1cm}\; = \| \mathbf{v}_{1:j} + \Delta \mathbf{v}_{1:j} + \Delta \mathbf{q}_{1:j} R_{1:j,1:j} \|_2 - c_7(\epsilon,n,\widehat{s})\|R_{1:j,1:j}\|_2\\
\nonumber
    &&\hspace{-1cm}\; \leq \| \mathbf{v}_{1:j}\|_2  + \|\Delta \mathbf{v}_{1:j}\|_2 + (\|\Delta \mathbf{q}_{1:j} \|_2 - c_7(\epsilon,n,\widehat{s}))\|R_{1:j,1:j}\|_2\\
    &&\hspace{-1cm}\; \leq \| \mathbf{v}_{1:j}\|_2  + \|\Delta \mathbf{v}_{1:j}\|_2 \label{eq:l4}\\
    &&\hspace{-1cm}\;\leq \| \mathbf{v}_{1:j}\|_2 + c_5(\epsilon,n,\widehat{s}) \| \mathbf{v}_{1:j}\|_2\label{eq:l5}\\
\nonumber
    &&\hspace{-1cm}\; = (1 + c_5(\epsilon,n,\widehat{s})) \| \mathbf{v}_{1:j}\|_2.
\end{eqnarray}
where the inequalities~\eqref{eq:l4} and \eqref{eq:l5} follows due to $\|\Delta \mathbf{q}_{1:j} \|_2 \le c_7(\epsilon,n,\widehat{s})$
and $\|\Delta \mathbf{v}_{1:j}\|_2 = c_5(\epsilon,n,\widehat{s}) \|\mathbf{v}_{1:j}\|_2$ from \eqref{eq:c7} and \eqref{eq:barlow1}, respectively.

Thus, we finally have
\begin{align}
    & \| \Delta \mathbf{v}_{1:j} + \Delta \mathbf{q}_{1:j} R_{1:j,1:j} \|_2 \nonumber \\&\leq  \| \Delta \mathbf{v}_{1:j}\|_2 + \|\Delta \mathbf{q}_{1:j}\|_2\|R_{1:j,1:j} \|_2 \nonumber \\
    &\leq \left( c_5(\epsilon,n,\widehat{s}) + c_7(\epsilon,n,\widehat{s}) \frac{1 + c_5(\epsilon,n,\widehat{s}}{1 - c_7(\epsilon,n,\widehat{s})}\right) \|\mathbf{v}_{1:j}\|_2 \nonumber \\
    &= c_8(\epsilon,n,\widehat{s}) \|\mathbf{v}_{1:j}\|_2,
\end{align}
where $c_8(\epsilon,n,\widehat{s}) = O(\epsilon)$. 

Therefore, there is an exactly-orthogonal matrix $\bar{\mathbf{q}}_{1:j}$ such that the backward error $\Delta \bar{\mathbf{v}}_{1:j}$ from BCGS2-HH satisfies
\begin{eqnarray}
\nonumber
    &&\mathbf{v}_{1:j} + \Delta \bar{\mathbf{v}}_{1:j} = \bar{\mathbf{q}}_{1:j}R_{1:j,1:j},\\
    &&\mbox{ with } \| \Delta \bar{\mathbf{v}}_{1:j}\|_2 = c_8(\epsilon,n,\widehat{s}) \|\mathbf{v}_{1:j}\|_2. \label{eq:vDelBound}
\end{eqnarray}
In contrast, if we compute the Householder QR of $\mathbf{v}_{1:j}$, the standard backward error analysis \cite[Theorem 19.4]{HighamNumAlg} gives 
\begin{eqnarray}
\nonumber
  &&\mathbf{v}_{1:j} + \Delta \mathbf{\tilde{v}}_{1:j} = \mathbf{\tilde{q}}_{1:j}\tilde{R}_{1:j,1:j},\\
\nonumber
  &&\mbox{ with } \|\Delta \tilde{\mathbf{v}}_{1:j} \|_2 \leq c_9(\epsilon, n, \widehat{s}) \|\mathbf{v}_{1:j}\|_2,
\end{eqnarray}
where $\tilde{R}_{1:j,1:j}$ is the upper-triangular matrix computed by the HH factorization in finite precision, $\mathbf{\tilde{q}}_{1:j}$ is exactly-orthogonal, and $c_9(\epsilon,n,\widehat{s}) = O(\epsilon)$ is a constant. In other words, the backward error from BCGS2-HH and Householder QR are both bounded by some $O(\epsilon)\|\mathbf{v}_{1:j}\|_2$ terms, and therefore they only differ by a constant factor $c(\epsilon, n, \widehat{s})$.
%
%
\end{proof}

\section*{Acknowledgment}
This work was supported 
by the Exascale Computing Project (17-SC-20-SC), a collaborative effort of the U.S. Department of Energy Office of Science and the National Nuclear Security Administration, 
by the U.S. Department of Energy, Office of Science, Office of Advanced Scientific Computing Research, Scientific Discovery through Advanced Computing (SciDAC) Program through the FASTMath Institute under Contract No. DE-AC02-05CH11231 at Sandia National Laboratories,
and by Sandia Laboratory Directed Research and Development (LDRD).
This research used resources of the National Energy Research Scientific Computing Center (NERSC).
Sandia National Laboratories is a multimission laboratory managed and operated by National Technology and Engineering Solutions of Sandia, LLC, a wholly owned subsidiary of Honeywell International, Inc., for the U.S. Department of Energy's National Nuclear Security Administration under contract DE-NA-0003525. This paper describes objective technical results and analysis. Any subjective views or opinions that might be expressed in the paper do not necessarily represent the views of the U.S. Department of Energy or the United States Government.

\bibliography{ref}
\end{document}